\documentclass[a4paper,10pt]{amsart}
\usepackage[margin=1in]{geometry}
\usepackage{enumerate, amsmath, amsfonts, amssymb, amsthm, mathtools, thmtools, wasysym, graphics, graphicx, xcolor, frcursive,xparse,comment,ytableau,stmaryrd,bbm,array,colortbl,tensor}
\usepackage[all]{xy}
\usepackage[boxed,norelsize]{algorithm2e}
\usepackage{hyperref}

\usepackage{url, hypcap}
\hypersetup{colorlinks=true, citecolor=darkblue, linkcolor=darkblue}

\usepackage{tikz}
\usetikzlibrary{calc,through,backgrounds,shapes,matrix}
\definecolor{darkblue}{rgb}{0.0,0,0.7} 
\newcommand{\darkblue}{\color{darkblue}} 
\definecolor{darkred}{rgb}{0.7,0,0} 
\definecolor{lightgrey}{rgb}{0.7,0.7,0.7} 

\definecolor{meet}{RGB}{255,205,111}
\definecolor{join}{RGB}{0,77,178}

\newtheorem{theorem}{Theorem}[section]
\newtheorem{proposition}[theorem]{Proposition}

\newtheorem{lemma}[theorem]{Lemma}

\theoremstyle{definition}
\newtheorem{definition}[theorem]{Definition}
\newtheorem{example}[theorem]{Example}

\newtheorem{remark}[theorem]{Remark}
\usepackage[procnames]{listings}
\usepackage[nameinlink]{cleveref}
\usepackage[all]{xy}
\usepackage[T1]{fontenc}
\crefformat{footnote}{#2\footnotemark[#1]#3}
\crefformat{conjecture}{Conjecture~#2#1#3}

\usepackage[cal=boondoxo]{mathalfa}
\usepackage[colorinlistoftodos]{todonotes}
\newcommand{\defn}[1]{\emph{\darkblue #1}}

\newcommand{\pop}{\mathsf{Pop}}
\newcommand{\cpop}{\mathsf{Pop}_{\lozenge}}

\usetikzlibrary{math}
\usepackage{xifthen}
\usepackage{xstring}
\SetKwInput{kwset}{set}
\usetikzlibrary{arrows,backgrounds,calc,trees}
\pgfdeclarelayer{background}
\pgfsetlayers{background,main}

\title{Crystal Pop-Stack Sorting and Type $A$ Crystal Lattices}

\author[C.~Defant]{Colin Defant}
\address[C.~Defant]{Princeton University}
\email{cdefant@princeton.edu}

\author[N.~Williams]{Nathan Williams}
\address[N.~Williams]{University of Texas at Dallas}
\email{nathan.williams1@utdallas.edu}

\subjclass[2000]{}

\begin{document}

\begin{abstract}
Given a complex simple Lie algebra $\mathfrak g$ and a dominant weight $\lambda$, let $\mathcal B_\lambda$ be the crystal poset associated to the irreducible representation of $\mathfrak g$ with highest weight $\lambda$. In the first part of the article, we introduce the \emph{crystal pop-stack sorting operator} $\cpop\colon\mathcal B_\lambda\to\mathcal B_\lambda$, a noninvertible operator whose definition extends that of the pop-stack sorting map and the recently-introduced Coxeter pop-stack sorting operators. Every forward orbit of $\cpop$ contains the minimal element of $\mathcal B_\lambda$, which is fixed by $\cpop$. We prove that the maximum size of a forward orbit of $\cpop$ is the Coxeter number of the Weyl group of $\mathfrak g$. In the second part of the article, we characterize exactly when a type $A$ crystal is a lattice.
\end{abstract}

\maketitle

\ytableausetup{smalltableaux}

\section{Introduction}

Let $\mathfrak S_n$ denote the symmetric group on $n$ letters, whose elements are permutations of the set $[n]:=\{1,\ldots,n\}$. Our story begins with the \defn{pop-stack sorting map}, the operator $\pop\colon \mathfrak S_{n}\to\mathfrak S_{n}$ that acts by reversing the descending runs (i.e., maximal consecutive decreasing subsequences) of a permutation. For example, the descending runs of $532481976$ are $532$, $4$, $81$, and $976$, so $\pop(532481976)=235418679$. Despite its simple definition, pop-stack sorting on permutations exhibits complicated dynamical properties. It is a deterministic variant of a pop-stack sorting machine introduced by Avis and Newborn \cite{Avis}, and it has recently received attention from a number of enumerative combinatorialists \cite{AlbertVatter, Asinowski, Asinowski2, Elder, ClaessonPop, ClaessonPop2, Pudwell}. 

Given a set $X$, a function $f\colon X\to X$, and an element $x\in X$, we define the \defn{forward orbit} of $x$ under $f$ to be the set $O_f(x)=\{x,f(x),f^2(x),\ldots\}$, where $f^i$ denotes the $i$-th iterate of $f$. Every forward orbit of a permutation under the map $\pop\colon\mathfrak S_n\to\mathfrak S_n$ contains the identity element $e$, which is fixed by $\pop$. Thus, for $w\in\mathfrak S_n$, $|O_\pop(w)|-1$ is the number of iterations of $\pop$ needed to ``sort'' $w$ into the identity permutation. The first appearance of the pop-stack sorting map was in a paper by Ungar on discrete geometry \cite{Ungar}, where he proved the surprisingly nontrivial fact that \[\max_{w\in\mathfrak S_n}|O_\pop(w)|=n.\]
This theorem was recently reproven by Albert and Vatter \cite{AlbertVatter}.

Suppose $W$ is a finite irreducible Coxeter group with Coxeter number $h$. The first author \cite{defant2021stack} defined the \defn{Coxeter pop-stack sorting operator} $\pop_W\colon W\to W$ by 
\begin{equation}\label{eqn:CoxPopDef1}
\pop_W(w)=w\cdot w_0(D_{\mathsf R}(w))^{-1},
\end{equation}
where $w_0(D_{\mathsf R}(w))$ is the longest element of the parabolic subgroup of $W$ generated by the right descent set of $w$. An equivalent definition is given by 
\begin{equation}\label{eqn:CoxPopDef2}
\pop_W(w)=\bigwedge(\{x\in W:x\lessdot_{\mathsf R} w\}\cup\{w\}),
\end{equation}
where the cover relations and the meet are taken in the right weak order on $W$. The latter definition extends naturally to arbitrary complete meet-semilattices, yielding the notion of a \defn{semilattice pop-stack sorting operator} that the first author explored in \cite{DefantMeeting}. The reason for using the name ``pop-stack sorting'' for these operators comes from the fact that $\pop_{\mathfrak S_n}$ coincides with the original pop-stack sorting map. In our previous article \cite{DefantWilliamsTorsing}, we defined and studied ``dual'' versions of the Coxeter pop-stack sorting operators that we called \emph{Coxeter pop-tsack torsing operators}.  

In all of these settings, one of the primary points of interest is the maximum size of a forward orbit of the operator in question. For example, in \cite{defant2021stack}, the first author proved a generalization of Ungar's theorem to an arbitrary finite irreducible Coxeter group $W$ by showing that 
\begin{equation}\label{eqn:CoxPopMain}
\max_{w\in W}|O_{\pop_W}(w)|=h,
\end{equation} where $h$ is the Coxeter number of $W$. Results of similar flavors were obtained for semilattice pop-stack sorting operators on $\nu$-Tamari lattices in \cite{DefantMeeting} and for (some) Coxeter pop-tsack torsing operators in \cite{DefantWilliamsTorsing}. 

\medskip 

Let $\mathfrak g$ be a complex semisimple Lie algebra. Given an indeterminate $q$, there is a \emph{quantum group} $U_q(\mathfrak g)$, which is obtained as a deformation of the universal enveloping algebra $U(\mathfrak g)$. Associated to each dominant weight $\lambda$ is a \defn{crystal}---a special finite poset $\mathcal B_\lambda$ that encodes crucial information about the irreducible representation $V^\lambda$ of $U_q(\mathfrak g)$ with highest weight $\lambda$. The underlying set of this poset is the crystal basis of $V^\lambda$, and the edges of the Hasse diagram of $\mathcal B_\lambda$ are colored in a manner that reflects the action of the Chevalley generators $e_i$ and $f_i$ of $U_q(\mathfrak g)$. The unique minimal element of $\mathcal B_\lambda$ is $v_\lambda$, the highest-weight vector of $V^\lambda$. There is a natural (right) action of the Weyl group $W$ of $\mathfrak g$ on $\mathcal B_\lambda$ (where the simple generator $s_i$ reverses all monochromatic saturated chains with color $i$), which gives an embedding of the right weak order on a parabolic quotient of $W$ into $\mathcal B_\lambda$ as the $W$-orbit of $v_\lambda$.  An example is given in~\Cref{fig:weak_crystal}. It is reasonable to ask what poset-theoretic properties are preserved under this embedding.  In~\cite{hersh2017weak}, Hersh and Lenart study this analogy from a poset-topological point of view.

In this article, we further generalize the Coxeter pop-stack sorting operators to the realm of crystals by defining the \emph{crystal pop-stack sorting operator} $\cpop\colon\mathcal B_\lambda\to\mathcal B_\lambda$. Suppose $b\in\mathcal B_\lambda$, and consider the set $b_\downarrow$ of colors of the edges in $\mathcal B_\lambda$ of the form $b'\lessdot b$. One can think of the colors in $b_\downarrow$ as the ``descents'' of $b$. Roughly speaking, $\cpop(b)$ is obtained by starting at $b$ and then walking down edges whose colors are in $b_\downarrow$ until one cannot walk down any further (see \Cref{subsec:crystals} for the formal definition). This definition mimics that of the Coxeter pop-stack sorting operators given in \eqref{eqn:CoxPopDef1}. Indeed, $\cpop$ coincides with $\pop_W$ on the aforementioned embedding of a parabolic quotient of $W$ into $\mathcal B_\lambda$. Every forward orbit of $\cpop$ contains $v_\lambda$, which is fixed by $\cpop$. Our main result about crystal pop-stack sorting operators extends \eqref{eqn:CoxPopMain}, which was one of the main theorems from \cite{defant2021stack}. 

\begin{theorem}\label{thm:popmain}
Let $\mathfrak g$ be a complex simple Lie algebra with Weyl group $W$. If $\lambda$ is a nonzero dominant weight in the weight lattice of $\mathfrak g$, then \[\max_{b\in\mathcal B_\lambda}|O_{\cpop}(b)|=h,\] where $h$ is the Coxeter number of $W$. 
\end{theorem}

If the crystal $\mathcal B_\lambda$ happens to be a lattice, then it comes equipped with two pop-stack sorting operators: the semilattice pop-stack sorting operator defined in \cite{DefantMeeting} and the crystal pop-stack sorting operator $\cpop$. These two operators need not coincide; for example, they differ on the type A crystal $\mathcal B_{(2,1)}^2$ depicted in \Cref{fig:weak_crystal}. In~\cite[Example 7.4]{hersh2017weak}, Hersh and Lenart observe that crystal posets are not always lattices---and since these two pop-stack sorting operators are not necessarily equal when a crystal is a lattice (an example of this difference is given by the crystal $\mathcal{B}_{(2,1)}^2$ on the right of~\Cref{fig:weak_crystal}), we are led to the problem of characterizing which crystals are lattices.

\tikzset{>=stealth}
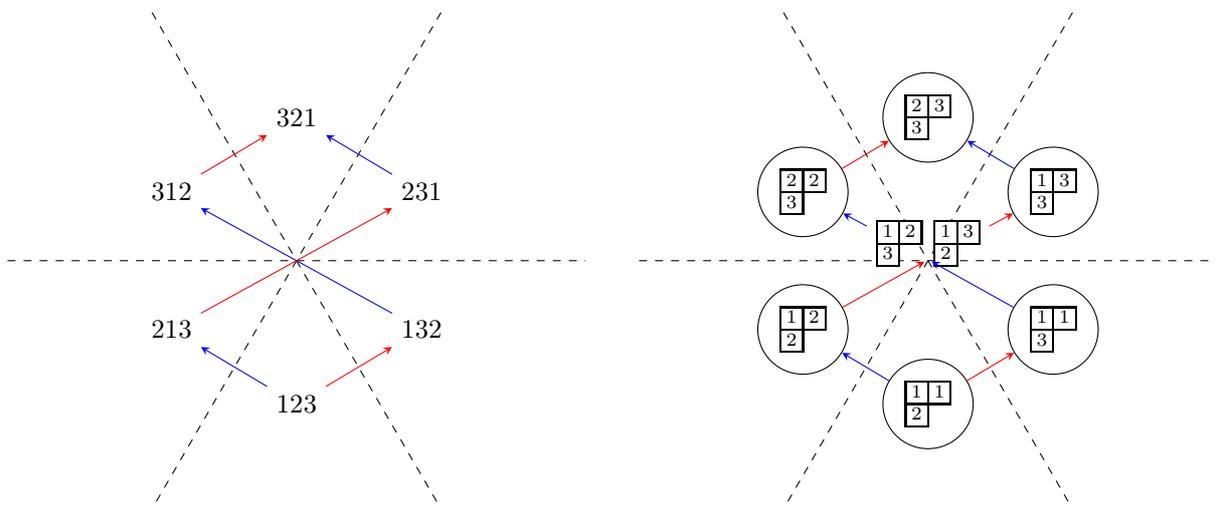
\begin{figure}[htbp]
\begin{tikzpicture}[scale=3.8]
\draw [dashed] (-1,0) -- (1,0);
\draw [dashed] (.5,.866) -- (-.5,-.866);
\draw [dashed] (-.5,.866) -- (.5,-.866);
\node (123) at (0,-.5) {123};
\node (132) at (.433,-.24) {132};
\node (213) at (-.433,-.24) {213};
\node (312) at (.433,.24) {231};
\node (231) at (-.433,.24) {312};
\node (321) at (0,.5) {321};
\draw [red,->] (123) -- (132);
\draw [blue,->] (132) -- (231);
\draw [red,->] (231) -- (321);
\draw [blue,->] (123) -- (213);
\draw [red,->] (213) -- (312);
\draw [blue,->] (312) -- (321);
\end{tikzpicture} \hfill
\begin{tikzpicture}[scale=3.8]
\draw [dashed] (-1,0) -- (1,0);
\draw [dashed] (.5,.866) -- (-.5,-.866);
\draw [dashed] (-.5,.866) -- (.5,-.866);
\node[circle,draw] (112) at (0,-.5) {\begin{ytableau} 1 & 1 \\ 2  \end{ytableau}};
\node[circle,draw] (113) at (.433,-.24) {\begin{ytableau} 1 & 1 \\ 3  \end{ytableau}};
\node[circle,draw] (122) at (-.433,-.24) {\begin{ytableau} 1 & 2 \\ 2  \end{ytableau}};

\node (123) at (-.1,.06) {\begin{ytableau} *(white) 1 & *(white)2 \\ *(white)3  \end{ytableau}};
\node (132) at (.1,.06) {\begin{ytableau} *(white)1 & *(white)3 \\ *(white)2  \end{ytableau}};

\node[circle,draw] (133) at (.433,.24) {\begin{ytableau} 1 & 3 \\ 3  \end{ytableau}};
\node[circle,draw] (223) at (-.433,.24) {\begin{ytableau} 2 & 2 \\ 3  \end{ytableau}};
\node[circle,draw] (233) at (0,.5) {\begin{ytableau} 2 & 3 \\ 3  \end{ytableau}};

\draw [red,->] (112) -- (113);
\draw [blue,->] (113) -- (123);
\draw [blue,->] (123) -- (223);
\draw [red,->] (223) -- (233);

\draw [blue,->] (112) -- (122);
\draw [red,->] (122) -- (132);
\draw [red,->] (132) -- (133);
\draw [blue,->] (133) -- (233);
\end{tikzpicture}
\caption{On the left is the right weak order on $\mathfrak{S}_3$.  On the right is the crystal $\mathcal{B}_{(2,1)}^2$.  The elements in the $\mathfrak{S}_3$-orbit of the highest-weight vector of $\mathcal{B}_{(2,1)}^2$ are circled.}
\label{fig:weak_crystal}
\end{figure}

\medskip

In this paper, we characterize exactly when $\mathcal B_\lambda^n$ is a lattice in the case of the complex simple Lie algebra $\mathfrak{sl}_{n+1}$ with Weyl group $\mathfrak S_{n+1}$. A partition $\lambda$ with at most $n$ parts can be viewed as a dominant weight for $\mathfrak{sl}_{n+1}$.  Note that we adopt the standard notation in which superscripts denote repeated parts in a partition (e.g., $(2^3,1)$ is shorthand for $(2,2,2,1)$).  To stress the dependence on $n$, we write $\mathcal B_\lambda^n$ for the associated crystal---such a crystal is said to be of type~$A_{n}$. Elements of $\mathcal B_\lambda^n$ are indexed by semistandard Young tableaux of shape $\lambda$ and maximum entry at most $n+1$.

\begin{theorem}\label{thm:main}
Every crystal of type $A_1$ or $A_2$ is a lattice. Suppose $n\geq 3$, and let $\lambda$ be a partition with at most $n$ parts. The crystal $\mathcal B_\lambda^n$ is a lattice if and only if at least one of the following conditions holds: 
\begin{itemize}
    \item $\lambda=(1^m)$ for some $0\leq m\leq n$;
    \item $\lambda=(2,1^m)$ for some $1\leq m\leq n-1$;
    \item $\lambda=(2^{n-m},1^m)$ for some $1\leq m\leq n-1$;
    \item $\lambda=(k)$ for some $k\geq 1$;
    \item $\lambda=(k^n)$ for some $k\geq 1$; 
    \item $\lambda=(k,1)$ for some $k\geq 1$;
    \item $\lambda=(k^{n-1},k-1)$ for some $k\geq 1$;
    \item $\lambda=(3,2,1)$ and $n=3$.
\end{itemize}
\end{theorem}

\begin{remark}
Since the weak order on a finite Weyl groups is always a lattice, it is natural to consider the problem of characterizing when crystals of other Cartan types are lattices.  We do not attempt such a characterization here.
\end{remark}

The organization of the paper is as follows. \Cref{sec:background} establishes background information on posets, Coxeter groups, crystals, and tableaux. We also define crystal pop-stack sorting operators (in a fairly general context) in \Cref{subsec:crystals}. In \Cref{sec:pop}, we prove \Cref{thm:popmain}. \Cref{sec:lattices} and \Cref{sec:notlattices} are devoted to proving \Cref{thm:main}.

\section{Background}\label{sec:background}

\subsection{Posets and Lattices}
We assume basic familiarity with the theory of posets, as outlined in \cite[Chapter~3]{StanleyEC1}. Suppose $P$ is a poset. For $x,y\in P$, we say $y$ \defn{covers} $x$ and write $x\lessdot y$ if $x<y$ and there does not exist $z\in P$ satisfying $x<z<y$. A \defn{saturated chain} in $P$ is a sequence of the form $x_1\lessdot x_2\lessdot\cdots\lessdot x_m$. An \defn{order ideal} of $P$ is a set $X\subseteq P$ such that if $y\in X$ and $x\leq y$, then $x\in X$. If two elements $x,y\in P$ have a greatest lower bound, then that element is called their \defn{meet} and is denoted $x\wedge y$. If $x$ and $y$ have a least upper bound, then that element is called their \defn{join} and is denoted $x\vee y$. The poset $P$ is a \defn{lattice} if $x\wedge y$ and $x\vee y$ exist for all $x,y\in P$. We write $\bigwedge X$ for the meet (i.e., greatest lower bound) of a set $X\subseteq P$. 

\subsection{Coxeter Groups}

We assume the reader is familiar with the basic aspects of Coxeter groups, which are treated in \cite{BjornerBrenti}. Suppose $(W,S)$ is a Coxeter system. We let $e$ denote the identity element of $W$. The \defn{length} of $w$, denoted $\ell(w)$, is the smallest length of a word $s_1\cdots s_k$ over the alphabet $S$ that, when viewed as a product of elements of $W$, equals $w$. A word of length $\ell(w)$ over $S$ that represents $w$ is called a \defn{reduced word} for $w$. Given $w_1,\ldots,w_r\in W$ with $w_1\cdots w_r=w$, we say the factorization $w=w_1\cdots w_r$ is \defn{length-additive} if $\ell(w)=\ell(w_1)+\cdots+\ell(w_r)$. A \defn{right descent} (respectively, \defn{left descent}) of $w$ is an element $s\in S$ such that $\ell(ws)<\ell(w)$ (respectively, $\ell(sw)<\ell(w)$). We write $D_{\mathsf R}(w)$ and $D_{\mathsf L}(w)$ for the set of right descents of $w$ and the set of left descents of $w$, respectively. There is a unique longest element of $W$, which we denote by $w_0$. 

The \defn{right weak order} on $W$ is the partial order $\leq_{\mathsf R}$ on $W$ with $x\leq_{\mathsf R} y$ if and only if $\ell(x^{-1}y)=\ell(y)-\ell(x)$. The cover relations in the right weak order are the relations of the form $ws\lessdot w$ with $s\in D_{\mathsf R}(w)$. The \defn{left weak order} on $W$ is the partial order $\leq_{\mathsf L}$ on $W$ with $x\leq_{\mathsf L} y$ if and only if $\ell(yx^{-1})=\ell(y)-\ell(x)$. The left and right weak orders are naturally isomorphic via the map $w\mapsto w^{-1}$. The \defn{Bruhat order} on $W$ is the partial order $\leq_{\mathsf B}$ on $W$ defined by saying that $x\leq y$ if some (equivalently, every) reduced word for $y$ contains a reduced word for $x$ as a (not necessarily contiguous) subword. The Bruhat order is an extension of the right weak order; this means that $x\leq_{\mathsf R} y$ implies $x\leq_{\mathsf B} y$. The right weak order for $\mathfrak{S}_3$ is illustrated on the left of~\Cref{fig:weak_crystal}. A fundamental result due to Bj\"orner \cite{BjornerWeak} states that the right weak order on a finite Coxeter group $W$ is a lattice.

For $J\subseteq S$, the \defn{parabolic subgroup} $W_J$ is the subgroup of $W$ generated by $J$. The longest element of $W_J$ is denoted by $w_0(J)$. For each $w\in W$, there is a unique representative of the coset $W_Jw$ that has minimum length among all elements of the coset; we denote this minimum-length representative by $\prescript{J}{}\!w$. The set $\prescript{J}{}\!W=\{\prescript{J}{}\!w:w\in W\}$ is called a \defn{parabolic quotient}; it has an alternative description as $\prescript{J}{}\!W=\{x\in W:D_{\mathsf L}(x)\subseteq S\setminus J\}$. We consider the right weak order $\leq_{\mathsf R}$ on $\prescript{J}{}\!W$, which is simply the order inherited from the right weak order on $W$. In fact, $\prescript{J}{}\!W$ is an order ideal of $W$ in the right weak order \cite[Proposition~2.5]{Stembridge}.

\subsection{Crystals}\label{subsec:crystals}
We will only provide a brief discussion of crystals, referring the reader to \cite{hersh2017weak, bump2017crystal, kashiwara1991crystal} for further details. Let us begin with a fairly abstract definition of a crystal. Let $I$ be a set; we view the elements of $I$ as colors. Let $B$ be a directed graph in which each edge is colored with exactly one of the elements of $I$. A directed path in $B$ is called \defn{monochromatic} if all of its edges have the same color. We say $B$ is a \defn{crystal} if it satisfies the following two properties:
\begin{itemize}
    \item Every monochromatic directed path in $B$ has finite length.
    \item For every color $i\in I$ and every vertex $b\in B$, there is at most one edge of the form $b'\to b$ with color $i$ and at most one edge of the form $b\to b''$ with color $i$.
\end{itemize}

Suppose $B$ is a crystal with color set $I$. Given a vertex $b$ of $B$, we write $b_\downarrow$ for the set of colors of edges of the form $b'\to b$. For $J\subseteq I$, let $B\vert_J$ be the crystal obtained from $B$ by deleting all edges whose colors are not in $J$. Recall that a \defn{source} of a directed graph is a vertex of in-degree $0$. 

\begin{definition}
Let $B$ be a crystal with color set $I$. We say $B$ is \defn{poppable} if for every $J\subseteq I$, every connected component of $B\vert_J$ has a unique source. If $B$ is poppable, then we define the \defn{crystal pop-stack sorting operator} $\cpop\colon B\to B$ by declaring $\cpop(b)$ to be the unique source of the connected component of $B\vert_{b_\downarrow}$ containing $b$.
\end{definition}

We now specialize our attention to crystals arising from irreducible representations of semisimple Lie algebras. Let $\mathfrak g$ be a complex semisimple Lie algebra with root system $\Phi$ and weight lattice $\Lambda$. We denote by $\Phi^+$ a choice of positve roots in $\Phi$ and by $\Lambda^+$ the associated set of dominant weights in $\Lambda$. Let $I$ be an indexing set such that $\{\alpha_i:i\in I\}\subseteq\Phi^+$ is the set of simple roots. For each $i\in I$, we write $s_i$ for the reflection through the hyperplace orthogonal to $\alpha_i$. The \defn{Weyl group} of $\mathfrak g$ is the group $W$ generated by the reflections $s_i$  for $i\in I$. The pair $(W,S)$ is a Coxeter system, where $S=\{s_i:i\in I\}$. 

For each $\lambda\in\Lambda^+$, let $V^\lambda$ be the irreducible representation of $\mathfrak g$ with highest weight $\lambda$. There is a unique (up to a nonzero scalar multiple) highest-weight vector $v_\lambda$, which is a vector in $V^\lambda$ with weight $\lambda$. Let $q$ be an indeterminate. There is a certain \emph{quantum group} $U_q(\mathfrak g)$ obtained as a deformation of the universal enveloping algebra $U(\mathfrak g)$. We can view $V^\lambda$ as a representation of $U_q(\mathfrak g)$. The \emph{Kashiwara operators} $E_i$ and $F_i$ are certain deformations of the Chevalley generators $e_i$ and $f_i$ that act on $V^\lambda$ in the $q\to 0$ limit, and the \defn{crystal basis} of $V^\lambda$ is obtained by taking the orbit of $v_\lambda$ under this action. 
This crystal basis forms the set of vertices of a crystal that we denote by $\mathcal B_\lambda$. This crystal is actually a finite poset, where each edge $b\to b'$ becomes a cover relation $b\lessdot b'$ (and these are all the cover relations). In what follows, these are the only types of crystals that we will consider, so we will often replace graph-theoretic language with the corresponding poset-theoretic language. Rather than define $\mathcal B_\lambda$ formally, let us simply state its properties that we will need and refer the reader to \cite{bump2017crystal, Kashiwara95oncrystal} for a more thorough treatment. 

First, $\mathcal B_\lambda$ has a unique minimal element and a unique maximal element; the minimal element is the highest-weight vector $v_\lambda$. The set of colors used to color the edges of $\mathcal B_\lambda$ is $I$ (the set that indexes the simple roots). Suppose $b\in \mathcal B_\lambda$ and $i\in I$. If there exists an edge of the form $b\lessdot b'$ with color $i$, then we define $F_i(b)=b'$ (this is well-defined by the definition of a crystal); otherwise, we set $F_i(b)=0$. This defines the \defn{lowering operator} $F_i\colon\mathcal B_\lambda\to\mathcal B_\lambda\cup\{0\}$. Similarly, we define the \defn{raising operator} $E_i\colon\mathcal B_\lambda\to\mathcal B_\lambda\cup\{0\}$ by letting $E_i(b)=b''$ if there exists an edge $b''\lessdot b$ with color $i$ and setting $E_i(b)=0$ otherwise. 

The Weyl group $W$ acts naturally on the weight space $\Lambda$. The stabilizer of $\lambda$ under this action is necessarily a parabolic subgroup of the form $W_K$ for some $K\subseteq S$. We will use the unusual convention that $W$ acts on $\Lambda$ on the right; thus, we compute $\mu\cdot (s_{i_1}s_{i_2}\cdots s_{i_r})$ by first applying the reflection $s_{i_1}$ to $\mu$, then applying the reflection $s_{i_2}$ to the resulting weight, and so on. There is also an interesting right action of $W$ on $\mathcal B_\lambda$ defined as follows. Suppose $b\in\mathcal B_\lambda$ and $i\in I$. Let $x_1\lessdot \cdots\lessdot x_m$ be the longest monochromatic saturated chain that contains $b$ and whose edges have color $i$. If $b=x_j$, then we define $b\cdot s_i$ to be $x_{m+1-j}$. Thus, $s_i$ sends the elements $x_1,\ldots,x_m$ to $x_m,\ldots,x_1$, respectively; we refer to this process as \defn{reversing an $i$-chain}. It is known that this definition extends to an action of all of $W$ on $\mathcal B_\lambda$~\cite[Section 11]{Kashiwara95oncrystal}. Furthermore, there is an injective map $\prescript{K}{}\!W\to\mathcal B_\lambda$ given by $w\mapsto v_\lambda\cdot w$. This map is actually a poset embedding of the right weak order on $\prescript{K}{}\!W$ into $\mathcal B_\lambda$; in other words, for $w,w'\in\prescript{K}{}\!W$, we have $w\leq_{\mathsf R}w'$ if and only if $v_\lambda\cdot w\leq v_\lambda\cdot w'$~\cite[Remark 2.12]{hersh2017weak}. We will often identify $\prescript{K}{}\!W$ with its image under this embedding.  

\begin{remark}
It might seem more natural to consider left actions of $W$ on $\Lambda$ and $\mathcal B_\lambda$ instead of right actions, and indeed, this is what is typically done in the literature on crystals. However, using a left action leads to an embedding of the left weak order on $(\prescript{K}{}\!W)^{-1}=\{w^{-1}:w\in\prescript{K}{}\!W\}$ instead of an embedding of the right weak order on $\prescript{K}{}\!W$. We have chosen to use right actions and the right weak order so that our conventions match with those used to define the pop-stack sorting map and the Coxeter pop-stack sorting operators in the introduction. This does not lead to any substantive issues because we can simply take inverses of Weyl group elements whenever we want to translate results from the crystal literature that are phrased in terms of left actions and the left weak order. \hfill $\triangle$
\end{remark}

Each element $b\in\mathcal B_\lambda$ is a weight vector in $V^\lambda$ with some weight $\text{wt}(b)\in\Lambda$; for $\mu\in\Lambda$, the set $\{b\in\mathcal B_\lambda:\text{wt}(b)=\mu\}$ is a basis for the weight space of $V^\lambda$ with weight $\mu$. In particular, $\{\text{wt}(b):b\in\mathcal B_\lambda\}$ is the set of weights for $V^\lambda$; it is well known that this set is equal to $\text{conv}(\lambda\cdot W)\cap\Lambda$, where $\text{conv}(\lambda\cdot W)$ is the convex hull of the $W$-orbit of $\lambda$. In fact, the elements of $\lambda\cdot W$ are the vertices of $\text{conv}(\lambda\cdot W)$.  By construction, if $b\in\mathcal B_\lambda$ and $i\in I$ are such that $F_i(b)\neq 0$, then $\text{wt}(F_i(b))=\text{wt}(b)-\alpha_i$. This leads to the following simple yet important fact: if $w\in W$ and $i\in I$, then \begin{equation}\label{eqn:EiFi=0}
    \text{either}\quad E_i(v_\lambda\cdot w)=0\quad\text{or}\quad F_i(v_\lambda\cdot w)=0.
\end{equation} Indeed, suppose this is not the case, and let $b=E_i(v_\lambda\cdot w)$ and $b'=F_i(v_\lambda\cdot w)$
with both $b\neq 0$ and $b'\neq 0$. Then $F_i(b)=v_\lambda\cdot W\neq 0$ and $F_i(v_\lambda\cdot W)=b'\neq 0$, so we have \[\text{wt}(v_\lambda\cdot W)=\text{wt}(b)-\alpha_i\quad\text{and}\quad \text{wt}(b')=\text{wt}(v_\lambda\cdot W)-\alpha_i.\] This is impossible because $\text{wt}(b)$ and $\text{wt}(b')$ are in $\text{conv}(\lambda\cdot W)$ and $\lambda\cdot W$ is a vertex of $\text{conv}(\lambda\cdot W)$. 

In the following lemma, recall that for $b\in\mathcal B_\lambda$, we write $b_\downarrow$ for the set of colors of edges of the form $b'\lessdot b$. As mentioned before, we identify each $w\in\prescript{K}{}\!W$ with its image $v_\lambda\cdot w\in\mathcal B_\lambda$ under the embedding $\prescript{K}{}\!W\to\mathcal B_\lambda$. 

\begin{lemma}\label{lem:descentsaredowncolors}
For $w\in\prescript{K}{}\!W$, we have $D_{\mathsf R}(w)=\{s_i:i\in w_\downarrow\}$.
\end{lemma}

\begin{proof}
Suppose $i\in w_\downarrow$. Then $E_i(w)\neq 0$. We saw in \eqref{eqn:EiFi=0} that $E_i(w)$ and $F_i(w)$ cannot both be nonzero, so $F_i(w)=0$. This shows that $w$ is at the top of some maximal monochromatic saturated chain of color $i$ that has at least two elements; the action of $s_i$ on $\mathcal B_\lambda$ reverses this $i$-chain, sending $w$ to the bottom element of the chain. Thus, $w\cdot s_i<w$ in $\mathcal B_\lambda$. This implies that $ws_i<_{\mathsf R}w$, so $s_i\in D_{\mathsf R}(w)$. 

To prove the reverse containment, suppose $s_i\in D_{\mathsf R}(w)$. Then $ws_i<_{\mathsf R} w$, so $ws_i<w$ in $\mathcal B_\lambda$. Referring again to the definition of the action of $s_i$ on $\mathcal B_\lambda$, we find that $w$ must not be the bottom element of the maximal monochromatic saturated chain of color $i$ that passes through $w$. That is, $i\in w_\downarrow$. 
\end{proof}

To complete this subsection, let us show that $\mathcal B_\lambda$ is actually a poppable crystal so that $\cpop\colon\mathcal B_\lambda\to\mathcal B_\lambda$ is well-defined. We will also show that the crystal pop-stack sorting operator $\cpop$ agrees with the Coxeter pop-stack sorting operator $\pop_W$ on $\prescript{K}{}\!W$. 

\begin{proposition}\label{lem:poppable}
Let $\mathfrak g$ be a semisimple Lie algebra with Weyl group $W$. Let $\lambda$ be a dominant weight for $\mathfrak g$, and let $K$ be the set of simple reflections of $W$ such that $W_K$ is the stabilizer of $\lambda$. The associated crystal $\mathcal B_\lambda$ is poppable. Furthermore, we have $\cpop(w)=\pop_W(w)$ for all $w\in \prescript{K}{}\!W$. 
\end{proposition}

\begin{proof}
Suppose $J\subseteq I$. Let $\mathfrak g_J$ be the semisimple Lie algebra whose Dynkin diagram is obtained from that of $\mathfrak g$ by taking the induced subgraph on the vertex set (corresponding to) $J$. Recall that $(\mathcal B_\lambda)\vert_J$ is the crystal obtained from $\mathcal B_\lambda$ by removing the edges whose colors are not in $J$. This process of removing edges is called \emph{Levi branching}, and it is known that each connected component of $(\mathcal B_\lambda)\vert_J$ is a crystal associated to an irreducible representation of $\mathfrak g_J$ (see \cite{Kashiwara95oncrystal}). As mentioned above, a crystal associated to an irreducible representation of a semisimple Lie algebra has a unique minimal element (the highest-weight vector). Therefore, every connected component of $(\mathcal B_\lambda)\vert_J$ has a unique minimal element (or, in graph-theoretic language, a unique source). Thus, $\mathcal B_\lambda$ is poppable. 

Now fix $w\in\prescript{K}{}\!W$. We know by \Cref{lem:descentsaredowncolors} that $s_i$ is a right descent of $w$ if and only if $i\in w_\downarrow$. We can compute $\cpop(w)$ as follows. First, choose some $i_1\in w_\downarrow$, and, starting from $w$, traverse down edges of color $i_1$ until it is no longer possible to do so; let $u_1$ be the resulting element of $\mathcal B_\lambda$. Since $E_{i_1}(w)\neq 0$, we must have $F_{i_1}(w)=0$ by \eqref{eqn:EiFi=0}. It follows that $u_1=w\cdot s_{i_1}$. Now find some $i_2\in(u_1)_{\downarrow}\cap w_\downarrow$, and traverse down edges of color $i_2$ until it is no longer possible to do so; let $u_2$ be the resulting element of $\mathcal B_\lambda$. Since $u_1\in\prescript{K}{}\!W$ and $E_{i_2}(u_1)\neq 0$, we must have $F_{i_2}(u_1)=0$ by \eqref{eqn:EiFi=0}. It follows that $u_2=u_1\cdot s_{i_2}=w\cdot(s_{i_1}s_{i_2})$. Continue in this fashion, at each step choosing $i_j\in (u_{j-1})_\downarrow\cap w_\downarrow$ until eventually reaching some $u_r$ with $(u_r)_\downarrow\cap w_\downarrow=\emptyset$. This element $u_r$ is $\cpop(w)$. Hence, $\cpop(w)=w\cdot(s_{i_1}s_{i_2}\cdots s_{i_r})$. Since $(u_r)_\downarrow\cap w_\downarrow=\emptyset$, \Cref{lem:descentsaredowncolors} tells us that $\cpop(w)$ and $w$ have no right descents in common. By construction, each $s_{i_j}$ is a right descent of $u_{j-1}$, so we have the length-additive factorization $w=\cpop(w)s_{i_1}s_{i_2}\cdots s_{i_r}$. On the other hand, the factorization $w=\pop_W(w)w_0(D_{\mathsf R}(w))$ given by \eqref{eqn:CoxPopDef1} is also length-additive (this is equivalent to the basic fact that $w_0(D_{\mathsf R}(w))\leq_{\mathsf L} w$). Each $i_j$ was chosen from $w_\downarrow$, so each $s_{i_j}$ is in $D_{\mathsf R}(w)$. It follows that there is a length-additive factorization $w_0(D_{\mathsf R}(w))=zs_{i_1}s_{i_2}\cdots s_{i_r}$ for some $z\in W_{D_{\mathsf R}(w)}$. Thus, we have $w=\cpop(w)s_{i_1}s_{i_2}\cdots s_{i_r}=\pop_W(w)zs_{i_1}s_{i_2}\cdots s_{i_r}$, where both factorizations are length-additive. This implies that $\pop_W(w)z$ is a length-additive factorization of $\cpop(w)$. Hence, every right descent of $z$ is also a right descent of $\cpop(w)$. But $z\in W_{D_{\mathsf R}(w)}$, so the right descents of $z$ are also right descents of $w$. We saw above that $\cpop(w)$ and $w$ have no right descents in common, so we must have $z=e$. Thus, $\cpop(w)=\pop_W(w)$. 
\end{proof}

\subsection{Partitions and Tableaux}
A \defn{partition} is a nonincreasing tuple of positive integers. We sometimes append $0$'s to the end of a partition, but doing so does not result in a different partition. We draw Young diagrams of partitions in English notation, writing $(i,j)$ for the cell in the $i$-th row (counted from top to bottom) and $j$-th column (counted from left to right). 

A \defn{semistandard Young tableau} (henceforth simply called a \defn{tableau}) of shape $\lambda$ is a filling of the cells of the Young diagram of $\lambda$ with positive integers such that rows are weakly increasing (from left to right) and columns are strictly increasing (from top to bottom). We write $T(i,j)$ for the entry in cell $(i,j)$ in the tableau $T$. The \defn{reading word} of a tableau is the word obtained by reading its entries row by row from bottom to top. For example, the reading word of 
\,\raisebox{0.5\height}{\begin{ytableau}
1 & 1 & 2 & 4 \\
2 & 4 \\
4
\end{ytableau}}\, is $4241124$. We often use superscripts to denote repeated entries in a row of a tableau. For example, we could describe \,\raisebox{0.5\height}{\begin{ytableau}
1 & 1 & 1 & 1 & 4 \\
2 & 4
\end{ytableau}}\, as the tableau of shape $(5,2)$ with rows $(1^4,4)$ and $(2,4)$. 

\subsection{Type A Crystals}
\label{sec:type_a}

In this subsection, we narrow our focus to the semisimple Lie algebra $\mathfrak{sl}_{n+1}$. The dominant weights are partitions of the form $\lambda=(\lambda_1,\ldots,\lambda_n,0)$, which we view as vectors in $\mathbb R^{n+1}$. Note that we allow such a partition $\lambda$ to have strictly fewer than $n$ nonzero parts (so not all of the parts $\lambda_1,\ldots,\lambda_n$ need to be positive). To stress the dependence on $n$, we write $\mathcal B_\lambda^n$ for the crystal of $\mathfrak{sl}_{n+1}$ associated to the dominant weight $\lambda$. The elements of $\mathcal B_\lambda^n$ are indexed by tableaux of shape $\lambda$ with maximum entry at most $n+1$; we will often tacitly identify $\mathcal B_\lambda^n$ with this set of tableaux. In particular, the minimal element of $\mathcal B_\lambda^n$, which is the tableaux corresponding to the highest-weight vector $v_\lambda$, is the tableau $T_{\min}$ of shape $\lambda$ in which $T_{\min}(i,j)=i$ for all cells $(i,j)$.

In \Cref{subsec:crystals}, we defined the lowering operators $F_i$ and the raising operators $E_i$ in general type; let us give an explicit combinatorial description of these operators for type~A crystals. Fix $i\in[n]$ and $T\in\mathcal B_\lambda^n$. We will actually just describe the procedure for applying $F_i$ since this description can easily be reversed in order to apply $E_i$ (and since we will really only use the lowering operators in our proofs). Consider the word $z_i$ obtained by deleting all letters from the reading word of $T$ that are not $i$ or $i+1$. Let $z_i'$ be the word obtained from $z_i$ by replacing each $i$ with a closing parenthesis $)$ and replacing each $i+1$ with an open parenthesis $($. Some of the open parentheses can match with closing parentheses in the obvious manner. If there are no unmatched closing parentheses in $z_i'$, then $F_i(T)=0$. Now suppose there is some unmatched closing parenthesis in $z_i'$. The rightmost such closing parenthesis corresponds to an occurrence of the letter $i$ in $z_i$, which, in turn, corresponds to an entry $i$ in the tableau $T$. Let $F_i(T)$ be the tableau obtained by changing that entry $i$ into $i+1$.

\begin{example}
Suppose $n=2$, $\lambda=(5,2)$, and \[T=\raisebox{0.5\height}{\begin{ytableau}
1 & 1 & 2 & 2 & 3 \\
3 & 3
\end{ytableau}}.\] Let us compute $F_1(T)$. By deleting all letters other than $1$ and $2$ from the reading word of $T$, we obtain the word $z_1=1122$. Then $z_1'=))(($, and the rightmost unmatched closing parenthesis in $z_1'$ is the second closing parenthesis, which corresponds to the second $1$ in $z_1'$. Thus, \[F_1(T)=\raisebox{0.5\height}{\begin{ytableau}
1 & 2 & 2 & 2 & 3 \\
3 & 3
\end{ytableau}}.\] On the other hand, when we apply $F_2$ to $T$, we find that $z_2=33223$ and $z_2'=(())($. Since $z_2'$ has no unmatched closing parentheses, $F_2(T)=0$. \hfill$\lozenge$
\end{example}

The cover relations of the partial order on $\mathcal B_\lambda^n$ are the relations of the form $T\lessdot F_i(T)$ for $i\in [n]$ and $F_i(T)\neq 0$. Each such edge is colored with the color $i$. When drawing the Hasse diagram of $\mathcal B_\lambda^n$, we will represent an edge colored $i$ with an upward-pointing arrow marked with the label $F_i$.

\section{Maximum Orbit Sizes for Crystal Pop-Stack Sorting}\label{sec:pop}

The purpose of this section is to prove \Cref{thm:popmain}. Let us fix a semisimple Lie algebra $\mathfrak g$ with root system $\Phi$, weight lattice $\Lambda$, and Weyl group $W$. We also fix a choice of positive roots $\Phi^+$ and an indexing set $I$ such that the simple roots are $\alpha_i$ for $i\in I$. Let $s_i\in W$ be the reflection through the hyperplane orthogonal to $\alpha_i$, and let $S=\{s_i:i\in I\}$. Let $h$ be the Coxeter number of $W$. We fix a nonzero dominant weight $\lambda\in\Lambda^+\setminus\{0\}$ and consider the corresponding crystal $\mathcal B_\lambda$. Let $K\subseteq S$ be the set of simple reflections such that $W_K$ is the stabilizer of $\lambda$ under the natural action of $W$. As discussed in \Cref{subsec:crystals}, the right weak order on the parabolic quotient $\prescript{K}{}\!W$ embeds into $\mathcal B_\lambda$; we identify $\prescript{K}{}\!W$ with its image under this embedding. 

Our main idea is to extract information about the crystal pop-stack sorting operator $\cpop\colon\mathcal B_\lambda\to\mathcal B_\lambda$ from the information about $\pop_W$ that the first author collected in \cite{defant2021stack}. The crucial tool for doing this is the \defn{key map} $\kappa\colon\mathcal B_\lambda\to \prescript{K}{}\!W$~\cite{littelmann1994littlewood}. This map is discussed in \cite{hersh2017weak}; rather than define it formally, let us simply record its properties that we will need in the following proposition. (Note that \cite{hersh2017weak} uses left actions and the left weak order, so we have translated to right actions and the right weak order by taking inverses.) 

\begin{proposition}[\cite{hersh2017weak}]\label{lem:key}
There exists a map $\kappa\colon\mathcal B_\lambda\to \prescript{K}{}\!W$ satisfying the following for all $i\in I$ and all $b\in \mathcal B_\lambda$:
\begin{itemize}
    \item If $E_i(b)\neq 0$ and $F_i(b)\neq 0$, then $\kappa(F_i(b))=\kappa(b)$.
    \item If $E_i(b)=0\neq F_i(b)$, then $\kappa(F_i(b))$ is either $\kappa(b)s_i$ or $\kappa(b)$.
    \item If $s_i$ is a right descent of $\kappa(b)$, then $E_i(b)\neq 0$.
    \item If $\kappa(b)=e$, then $b=v_\lambda$.
\end{itemize}
These properties imply that $\kappa$ is order-preserving, meaning that if $b\leq b'$, then $\kappa(b)\leq_{\mathsf R}\kappa(b')$.
\end{proposition}

We will frequently use the key map to project to the parabolic quotient $\prescript{K}{}\!W$ and then apply the Coxeter pop-stack sorting operator $\pop_W$. Since $\prescript{K}{}\!W$ is an order ideal of the right weak order on $W$ and $\pop_W(w)\leq_{\mathsf R} w$ for all $w\in W$, the map $\pop_W$ does actually restrict to a well-defined map from $\prescript{K}{}\!W$ to itself. We will need the following results from \cite{defant2021stack}. 

\begin{lemma}[\cite{defant2021stack}]\label{lem:Lem2.1}
Suppose $J\subseteq S$. If $y,z\in W$ are such that $y\leq_{\mathsf R} z$, then $\prescript{J}{}\!y\leq_{\mathsf R} \prescript{J}{}\!z$. 
\end{lemma}

\begin{lemma}[\cite{defant2021stack}]\label{lem:Lem3.5}
Let $x,y\in W$, and suppose all of the right descents of $y$ commute with each other. If $x\leq_{\mathsf B} y$, then $\pop_W(x)\leq_{\mathsf B}\pop_W(y)$.
\end{lemma}

\begin{lemma}[\cite{defant2021stack}]\label{lem:3.6}
If $J\subseteq S$, then $\prescript{J}{}\!(\pop_W(w))\leq_{\mathsf R}\pop_W(\prescript{J}{}\!w)$ for all $w\in W$.
\end{lemma}

The following proposition was proven in \cite{defant2021stack} for the case when the Coxeter number $h$ is even. The only Weyl groups with odd Coxeter numbers are the symmetric groups $\mathfrak S_{n}$ when $n$ is odd. In these cases, the proposition follows from the work of Ungar \cite{Ungar} (and a short additional argument). 

\begin{lemma}[\cite{defant2021stack, Ungar}]\label{lem:defantw0}
Choose $s\in S$, and let $J=S\setminus\{s\}$. We have $\pop_W^{h-1}(\prescript{J}{}\!w_0)=e$ and $\pop_W^{h-2}(\prescript{J}{}\!w_0)\neq e$. Furthermore, for all $t\geq 0$, the right descents of $\pop_W^t(\prescript{J}{}\!w_0)$ all commute with each other.
\end{lemma}

\begin{proof}
As mentioned above, this result was proven in \cite{defant2021stack} when $h$ is even, so we may assume $h$ is odd. This means that $W=\mathfrak S_n$, where $h=n$ is odd. Also, $\pop_W$ is just the pop-stack sorting map $\pop$, which acts on permutations by reversing descending runs in the manner discussed in the introduction. Ungar \cite{Ungar} proved that $\pop^{n-1}(w)=e$ for all $w\in\mathfrak S_n$ (so, in particular, when $w=\prescript{J}{}\!w_0$). He also proved that $\pop_W^{n-2}(\prescript{J}{}\!w_0)\neq e$ (in fact, he proved that this holds whenever $J$ is \emph{any} nonempty proper subset of $S$, not just one of the form $S\setminus\{s\}$). Thus, we are left to show that if $t\geq 0$, then all of the right descents of $\pop^t(\prescript{J}{}\!w_0)$ commute with each other. Fix $t\geq 0$, and let $v=\pop^t(\prescript{J}{}\!w_0)$. Let $i\in[n-1]$ be such that $s$ is the simple transposition $(i\,\, i+1)$. Since $\prescript{J}{}\!W$ is an order ideal of the right weak order on $W$, we know that $v\in\prescript{J}{}\!W$. This means that $v$ has no left descents in $J$, so 
\begin{equation}\label{eqn:vstructure}
v^{-1}(1)<v^{-1}(2)<\cdots<v^{-1}(i)\quad\text{and}\quad v^{-1}(i+1)<v^{-1}(i+2)<\cdots<v^{-1}(n).
\end{equation} 
Saying the right descents of $v$ commute with each other is equivalent to saying there does not exist $m\in[n-2]$ such that $v(m)>v(m+1)>v(m+2)$. That no such $m$ exists is immediate from \eqref{eqn:vstructure}. 
\end{proof}

We will also need the following simple lemma. 
\begin{lemma}\label{lem:technical}
Suppose $w\in\prescript{K}{}\!W$ is such that $\prescript{S\setminus\{s\}}{}\!w=e$ for all $s\in S\setminus K$. Then $w=e$.
\end{lemma}

\begin{proof}
The equation $\prescript{S\setminus \{s\}}{}\!w=e$ is equivalent to the containment $w\in W_{S\setminus \{s\}}$. This shows that $w\in\bigcap_{s\in S\setminus K}W_{S\setminus\{s\}}=W_K$, so $w\in W_K\cap \prescript{K}{}\!W=\{e\}$.
\end{proof}

One of our tools for leveraging the key map $\kappa$ comes from the following lemma, which helps us understand how it interacts with crystal pop-stack sorting. If $b\in\mathcal B_\lambda$, then (since $\cpop(\prescript{K}{}\!W)\subseteq \prescript{K}{}\!W$) both $\kappa(\cpop(b))$ and $\cpop(\kappa(b))$ are in $\prescript{K}{}\!W$, so it makes sense to compare them in the right weak order.

\begin{lemma}\label{lem:popkey}
For every $b\in\mathcal B_\lambda$, we have $\kappa(\cpop(b))\leq_{\mathsf R}\cpop(\kappa(b))$. 
\end{lemma}

\begin{proof}
Suppose first that $\kappa(b)=e$. Under the embedding of $\prescript{K}{}\!W$ into $\mathcal B_\lambda$, the identity $e$ gets identified with the minimal element $v_\lambda$. We know by \Cref{lem:key} that $\kappa$ is order-preserving, and we know that $\cpop(b)\leq b$. Therefore, $\kappa(\cpop(b))\leq_{\mathsf R}\kappa(b)=e\leq_{\mathsf R}\cpop(\kappa(b))$.

We may now assume that $\kappa(b)\neq e$. Since $\kappa(b)\in\prescript{K}{}\!W$, we know by \Cref{lem:poppable} and \eqref{eqn:CoxPopDef2} that \[\cpop(\kappa(b))=\pop_W(\kappa(b))=\bigwedge\{x\in W:x\lessdot_{\mathsf R} \kappa(b)\},\] where the cover relations and meet are in the right weak order. Hence, it suffices to show that $\kappa(\cpop(b))\leq_{\mathsf R} x$ for every $x\in W$ that is covered by $\kappa(b)$ in the right weak order. Let us fix such an element $x$. 

There exists $i\in I$ such that $s_i\in D_{\mathsf R}(\kappa(b))$ and $x=\kappa(b)s_i$. There is a monochromatic saturated chain $b_0\lessdot b_1\lessdot\cdots\lessdot b_m$ of color $i$ in $\mathcal B_\lambda$ such that $b_m=b$ and $E_i(b_0)=0$. Because $s_i$ is a right descent of $\kappa(b)$, \Cref{lem:key} tells us that $E_i(b)\neq 0$. Hence, $b_0\neq b$ (i.e., $m\geq 1$). For each $j\in[m-1]$, we have $E_i(b_j)\neq 0$ and $F_i(b_j)\neq 0$, so it follows from \Cref{lem:key} that $\kappa(b_j)=\kappa(F_i(b_j))=\kappa(b_{j+1})$. Consequently, $\kappa(b_1)=\kappa(b)$. We have $E_i(b_0)=0$ and $F_i(b_0)=b_1\neq 0$, so \Cref{lem:key} tells us that $\kappa(b_1)$ is either $\kappa(b_0)s_i$ or $\kappa(b_0)$. If $\kappa(b_1)$ were equal to $\kappa(b_0)$, then we would have $s_i\in D_{\mathsf R}(\kappa(b))=D_{\mathsf R}(\kappa(b_1))=D_{\mathsf R}(\kappa(b_0))$. However, \Cref{lem:key} would then tell us that $E_i(b_0)\neq 0$, which would be a contradiction. This shows that $\kappa(b_1)=\kappa(b_0)s_i$. Therefore, $x=\kappa(b)s_i=\kappa(b_1)s_i=\kappa(b_0)$. 

Recall that our goal is to prove that $\kappa(\cpop(b))\leq_{\mathsf R} x$. Since $x=\kappa(b_0)$ and $\kappa$ is order-preserving (by \Cref{lem:key}), it suffices to show that $\cpop(b)\leq b_0$. However, this is clear from the definition of $\cpop$. Indeed, $i\in b_\downarrow$ because the edge $b_{m-1}\lessdot b$ has the color $i$. Since there is a monochromatic saturated chain of color $i$ connecting $b_0$ to $b$, we know that $b_0$ is in the same connected component of $\mathcal B_\lambda\vert_{b_\downarrow}$ as $b$. The unique minimal element of this connected component is $\cpop(b)$ by definition, so $\cpop(b)\leq b_0$. 
\end{proof}

We are now in a position to demonstrate that $\max\limits_{b\in\mathcal B_\lambda}|O_{\cpop}(b)|=h$, thereby proving \Cref{thm:popmain}.

\begin{proof}[Proof of \Cref{thm:popmain}]
As before, let $W_K$ be the stabilizer of $\lambda$ in $W$ so that $\prescript{K}{}\!W$ embeds into $\mathcal B_\lambda$ as the orbit of $v_\lambda$. Since $\lambda$ is nonzero, it is not fixed by the entire Weyl group $W$. This implies that $W_K$ is not all of $W$. Choose some $s\in S\setminus K$, and let $J=S\setminus\{s\}$. The parabolic quotient $\prescript{J}{}\!W$ is contained in $\prescript{K}{}\!W$. Since $\prescript{J}{}\!W$ and $\prescript{K}{}\!W$ are order ideals of the right weak order on $W$, $\prescript{J}{}\!W$ is actually an order ideal of the right weak order on $\prescript{K}{}\!W$. Therefore, $\cpop$ restricts to a map $\prescript{K}{}\!W\to\prescript{K}{}\!W$ (which agrees with $\pop_W$), which further restricts to a map $\prescript{J}{}\!W\to\prescript{J}{}\!W$. Appealing to \Cref{lem:defantw0}, we find that \[\max_{b\in\mathcal B_\lambda}|O_{\cpop}(b)|\geq\max_{w\in \prescript{J}{}\!W}|O_{\pop_W}(w)|\geq |O_{\pop_W}(\prescript{J}{}\!w_0)|=h.\] We are left to show that $\cpop^{h-1}(b)=v_\lambda$ for all $b\in\mathcal B_\lambda$. Let us fix $b\in\mathcal B_\lambda$. We are going to prove that $\prescript{J}{}\!(\kappa(\cpop^{h-1}(b)))=e$. Since $s$ was chosen arbitrarily from $S\setminus K$ and $J=S\setminus\{s\}$, it will then follow from \Cref{lem:technical} that $\kappa(b)=e$. This, in turn, will imply that $b=v_\lambda$ by the fourth bullet point in \Cref{lem:key}. 

In what follows, we will always use the symbol $\cpop$, keeping in mind that this map agrees with $\pop_W$ on $\prescript{K}{}\!W$. Our goal is to prove that $\prescript{J}{}\!(\kappa(\cpop^{h-1}(b)))=e$. We will actually prove by induction on $t$ that $\prescript{J}{}\!(\kappa(\cpop^t(b)))\leq_{\mathsf B}\cpop^t(\prescript{J}{}\!w_0)$ for all $t\geq 0$. Since we know by \Cref{lem:defantw0} that $\cpop^{h-1}(\prescript{J}{}\!w_0)=e$, this will complete the proof. 

When $t=0$, we just need to show that $\prescript{J}{}\!(\kappa(b))\leq_{\mathsf B}\prescript{J}{}\!w_0$, which is clear from \Cref{lem:Lem2.1}. Now suppose $t\geq 1$, and assume we have already proven that $\prescript{J}{}\!(\kappa(\cpop^{t-1}(b)))\leq_{\mathsf B}\cpop^{t-1}(\prescript{J}{}\!w_0)$. The right descents of $\cpop^{t-1}(\prescript{J}{}\!w_0)$ all commute with each other by \Cref{lem:defantw0}, so we can set $x=\prescript{J}{}\!(\kappa(\cpop^{t-1}(b)))$ and $y=\cpop^{t-1}(\prescript{J}{}\!w_0)$ in \Cref{lem:Lem3.5} to find that 
\begin{equation}\label{Eq1}
\cpop(\prescript{J}{}\!(\kappa(\cpop^{t-1}(b))))\leq_{\mathsf B}\cpop^t(\prescript{J}{}\!w_0).
\end{equation} We know by \Cref{lem:popkey} that $\kappa(\cpop^t(b))\leq_{\mathsf R}\cpop(\kappa(\cpop^{t-1}(b)))$, so it follows from \Cref{lem:Lem2.1} that $\prescript{J}{}\!(\kappa(\cpop^t(b)))\leq_{\mathsf R}\prescript{J}{}\!(\cpop(\kappa(\cpop^{t-1}(b))))$. According to \Cref{lem:3.6}, we have $\prescript{J}{}\!(\cpop(\kappa(\cpop^{t-1}(b))))\leq_{\mathsf R}\cpop(\prescript{J}{}\!(\kappa(\cpop^{t-1}(b))))$. This shows that $\prescript{J}{}\!(\kappa(\cpop^t(b)))\leq_{\mathsf R}\cpop(\prescript{J}{}\!(\kappa(\cpop^{t-1}(b))))$, so $\prescript{J}{}\!(\kappa(\cpop^t(b)))\leq_{\mathsf B}\cpop(\prescript{J}{}\!(\kappa(\cpop^{t-1}(b))))$ because the Bruhat order is an extension of the right weak order. Combining this with \eqref{Eq1} yields $\prescript{J}{}\!(\kappa(\cpop^t(b)))\leq_{\mathsf B}\cpop^t(\prescript{J}{}\!w_0)$. This completes the induction.
\end{proof}

\section{Type A Crystals That Are Lattices}\label{sec:lattices}

Fix $n \geq 1$.  Let $\lambda=(\lambda_1,\lambda_2,\ldots,\lambda_{n},0)$ be a dominant weight for $\mathfrak{sl}_{n+1}$ (i.e., a partition with at most $n$ nonzero parts) as in \Cref{sec:type_a}. In this section, we prove that $\mathcal B_\lambda^n$ is a lattice whenever $n$ and $\lambda$ satisfy at least one of the conditions listed in \Cref{thm:main}.  It is useful to introduce the notation  $J(P)$ for the distributive lattice of order ideals of a finite poset $P$.  We can view $[a]$ as the $a$-element chain poset.

We write $\lambda^* = (\lambda_1,\lambda_1-\lambda_{n},\lambda_1-\lambda_{n-1},\ldots,\lambda_1-\lambda_2,0)$.  Associated to the type $A_n$ crystal $\mathcal B_\lambda^n$ is the \defn{dual crystal} $(\mathcal B_\lambda^n)^*$, which is defined to be the crystal obtained from $\mathcal B_\lambda^n$ by replacing each edge color $i$ by $n+1-i$ (so $\mathcal B_\lambda^n$ and $(\mathcal B_\lambda^n)^*$ have the same directed graph structure when one ignores the edge colors).  The following shows that the representation underlying the dual crystal is the dual representation.
\begin{proposition}
\label{prop:dual}
If $\lambda$ is a dominant weight for $\mathfrak{sl}_{n+1}$, then the crystals $\left(\mathcal B_\lambda^n\right)^*$ and $\mathcal B_{\lambda^*}^n$ are isomorphic. 
\end{proposition}
\begin{proof}
For any Cartan type, the dual of $V^\lambda$ is isomorphic to $V^{-\lambda\cdot w_0}$, where $w_0$ is the longest element of the Weyl group~\cite[Proposition 21.1.2]{lusztig2010introduction}.  Restricting to type $A_n$ and renormalizing, we obtain $-\lambda\cdot w_0=-(\lambda_1,\lambda_2,\ldots,\lambda_{n},0)\cdot w_0 = (\lambda_1,\lambda_1-\lambda_{n},\lambda_1-\lambda_{n-1},\ldots,\lambda_1-\lambda_2,0) = \lambda^*$.  Upgrading to the crystals, the action of $F_i$ becomes the action of $F_{n+1-i}$ under the map $-w_0$, which explains the relabeling of the edges.
\end{proof}

\subsection{\texorpdfstring{$A_1$ and $A_2$}{A1 and A2}}
\label{sec:a1_and_a2}
The representation theory of $\mathfrak{sl}_2$ is classical, with highest-weight representations indexed by nonnegative integers $k$.  The elements of the crystal $\mathcal B_{(k)}^1$ correspond to semistandard tableaux with a single row of $k$ boxes filled with 1's and 2's. The action of $F_1$ simply converts the rightmost 1 to a 2, so that $A_1$ crystals are chains, and hence are lattices.

Using a graph-theoretic construction, it was shown in~\cite[Proposition 5.4]{danilov2007combinatorics} that $A_2$ crystals all have a particular product representation, from which it follows that they are lattices.

\subsection{\texorpdfstring{$\lambda=(1^m)$}{Minuscules}}
\label{sec:minuscule}
The stabilizer of $\lambda$ in $\mathfrak S_{n+1}$ is $(\mathfrak S_{n+1})_K$, where $K=[n]\setminus\{m\}$. The crystal $\mathcal{B}_{(1^m)}^n$ corresponds to a minuscule representation, and it is well known to be isomorphic to the distributive lattice $J([m]\times [n+1-m])$.  Indeed, because $V^{(1^m)}$ is minuscule, the weights in the representation appear with multiplicity $1$ and all lie in the $\mathfrak{S}_{n+1}$-orbit of $\lambda=(1^m)$---thus, for any $v \in \mathcal{B}_{(1^m)}^n$, we have $v\cdot s_i = F_i(v)$ if $F_i(v)\neq 0$ and $v\cdot s_i=v$ otherwise.  This shows that the embedding $\prescript{K}{}\!\mathfrak S_{n+1}\to\mathcal B_{(1^m)}^n$ discussed in \Cref{subsec:crystals} is an isomorphism of posets. In other words, $\mathcal B_{(1^m)}^n$ is isomorphic to the weak order on the set of permutations $w\in \mathfrak S_{n+1}$ satisfying $D_{\mathsf L}(w)\subseteq\{m\}$; this is easily seen to be isomorphic to the distributive lattice $J([m]\times [n+1-m])$.

More directly, thinking of elements of $J([m]\times [n+1-m])$ as partitions fitting inside an $m\times (n+1-m)$ box, an isomorphism of posets is given by mapping the tableau \[{\ytableausetup{boxsize=1.2em}\begin{ytableau}
a_1 \\ a_2 \\ \raisebox{-.08cm}{\vdots} \\ a_m
\end{ytableau}}\] in $\mathcal B_{(1^m)}^{n}$ to the partition $(a_m-m,\cdots,a_2-2,a_1-1)$; hence, the tableaux are ordered by componentwise comparison.

\subsection{\texorpdfstring{$\lambda=(k)$ or $(k^n)$}{Multiples of vector representation}}
\label{sec:mult_first_fund_weight}
By~\Cref{prop:dual}, the two crystals $\mathcal B_{(k)}^n$ and $\mathcal B_{(k^n)}^n$ are dual, so they are isomorphic as unlabeled directed graphs.  We therefore need only consider $\lambda=(k)$. We will define a poset isomorphism between the poset underlying $\mathcal B_{(k)}^{n}$ and the distributive lattice $J([k]\times [n])$.  Thinking of elements of $J([k]\times [n])$ as partitions fitting inside a $k\times n$ box, a tableau \raisebox{-.25\height}{\ytableausetup{boxsize=1.2em}\begin{ytableau}
a_1 & a_2 & \cdots & a_k
\end{ytableau}} in $\mathcal B_{(k)}^{n}$ maps to the partition $(a_k-1,\cdots,a_2-1,a_1-1)$.  Since the partition has a single row, there is no pairing of $i$'s and $(i+1)$'s when computing the action of $F_i$---thus, $F_i$ changes the rightmost $i$ to an $i+1$, which exactly corresponds to a cover in $J([k]\times [n])$.  Note that these tableaux are ordered by componentwise comparison: $\raisebox{-.25\height}{\ytableausetup{boxsize=1.2em}\begin{ytableau}
a_1 & a_2 & \cdots & a_k
\end{ytableau}}\leq \raisebox{-.25\height}{\ytableausetup{boxsize=1.2em}\begin{ytableau}
b_1 & b_2 & \cdots & b_k
\end{ytableau}}$ in $\mathcal B_{(k)}^{n}$ if and only if $a_i \leq b_i$ for all $1 \leq i \leq k$.

\ytableausetup
 {mathmode, boxsize=2em}
\subsection{\texorpdfstring{$\lambda=(2,1^m)$ or $(2^{n-m},1^m)$}{Long hooks}}
By~\Cref{prop:dual}, the two crystals $\mathcal B_{(2,1^m)}$ and $\mathcal B_{(2^{n-m},1^m)}$ are dual, so we consider only $\lambda=(2,1^m)$.  We will show that any two elements of $\mathcal B_\lambda^n$ have a join; since the crystal has a unique minimal element, this will imply that it is a lattice. Let $T,T'$ be two given tableaux in $\mathcal B_\lambda^n$.  If $n+1$ does not appear in $T$ or $T'$, then the interval from \[\raisebox{0.5\height}{\begin{ytableau}
1 & 1 \\
2 \\
\raisebox{-.06cm}{\vdots} \\
\scalebox{.8}{{\it m}+1}
\end{ytableau}} \hspace{2em} \text{ to }  \hspace{2em} \raisebox{0.5\height}{\begin{ytableau}
\scalebox{.8}{{\it n{-}m}} & n \\
\scalebox{.6}{{\it n}{-}{\it m}{+}1} \\
\raisebox{-.06cm}{\vdots} \\
n\end{ytableau}} \]
in $\mathcal B_{(2,1^m)}^n$ matches the same interval in $\mathcal B_{(2,1^m)}^{n-1}$, and $T$ and $T'$ both appear within this interval.  By induction on $n$ (the base case when $m=n$ is handled by~\Cref{sec:minuscule} because $\mathcal B_{(2,1^n)}^n$ is isomorphic to $\mathcal{B}_{(1)}^n$ since the first column contains all numbers $1,2,\ldots,n+1$), this interval is a lattice, so we can find the join of $T$ and $T'$.   

Now assume that $n+1$ appears in $T$ or $T'$. Let $a=T(1,2)$, $a'=T'(1,2)$, $b=T(m+1,1)$, and $b'=T(m+1,1)$.  One of $a,a',b,b'$ must be equal to $n+1$. By symmetry, we may assume that either $a$ or $b$ is $n+1$. 

\begin{itemize}
    \item Suppose $a=n+1$ and $b,b'\neq n+1$. Consider the interval between \[A^{(0)}=\raisebox{0.5\height}{\begin{ytableau}
1 & \scalebox{.8}{{\it n}+1} \\
2 \\
\vdots \\
\scalebox{.8}{{\it m}+1}
\end{ytableau}}\hspace{2em}\text{ and }\hspace{2em} A^{(1)}=\raisebox{0.5\height}{\begin{ytableau}
\scalebox{.8}{{\it n{-}m}}  & \scalebox{.8}{{\it n}+1} \\
\scalebox{.6}{{\it n}{-}{\it m}{+}1} \\
\vdots \\
n
\end{ytableau}}.\]
This interval $[A^{(0)},A^{(1)}]$ is isomorphic to the crystal $\mathcal B_{(1^{m+1})}^{n-1}$ addressed in~\Cref{sec:minuscule} (and hence tableaux are ordered by component-wise comparison), so it is a lattice.  Note that $T\in[A^{(0)},A^{(1)}]$.
\begin{itemize}
\item If $a'\geq b'$, then the tableau $T''=F_n \cdots F_{a'+1}F_{a'} (T')$ is the smallest element of $[A^{(0)},A^{(1)}]$ greater than $T'$, so the least upper bound of $T$ and $T'$ is $T \vee T''$, the join of $T$ and $T''$ in $[A^{(0)},A^{(1)}]$.  By~\Cref{sec:minuscule}, this is also the tableau given by component-wise maximum of the first columns of $T$ and $T'$ that has entry $n+1$ in the cell $(1,2)$.

\item If $a'< b'$, then any upper bound $X$ for $A_0$ and $T'$ must have $X(1,2)=n+1$.  Let us consider how to start with $T'$ and increase the entry in cell $(1,2)$ from $a'$ to $n+1$.  For each $i=a',a'+1,\ldots,n$, in any path from $T'$ to $X$, we must at some point apply $F_i$ to increase the entry in cell $(1,2)$ from $i$ to $i+1$.  For this $F_i$ to have any effect on the entry in cell $(1,2)$, any entry equal to $i-1$ in the first row of the tableau needs to have already been increased to $i$.  Thus, at some point, the entry in cell $(m+1,1)$ must be increased to $n+1$.  In fact, since $a'< b'$, we can increase the entry in cell $(m+1,1)$ to $n+1$ without affecting any other entries of $T'$ to produce $T''=F_n \cdots F_{b'+1}F_{b'}(T')$---and any sequence of $F_i$ operators leading to $X$ could be prefaced by this initial sequence.  We can therefore reduce to the case when $b'=n+1$. By symmetry, we can handle the case when $b=n+1$ instead.
\end{itemize}

\item If $b=n+1$, then we consider the interval between
\[B^{(0)}=\raisebox{0.5\height}{\begin{ytableau}
1 & 1 \\
2 \\
\vdots \\
m\\
\scalebox{.8}{{\it n}+1}
\end{ytableau}}\hspace{2em}\text{ and }\hspace{2em} B^{(1)}=\raisebox{0.5\height}{\begin{ytableau}
\scalebox{.8}{{\it n{-}m}}  & \scalebox{.8}{{\it n}+1} \\
\scalebox{.6}{{\it n}{-}{\it m}{+}1} \\
\vdots \\
n\\
\scalebox{.8}{{\it n}+1}
\end{ytableau}}.\]
This interval $[B^{(0)},B^{(1)}]$ is isomorphic to $\mathcal B_{(2,1^{m-1})}^n$, so it is a lattice by induction on $m$.  Note that $T\in [B^{(0)},B^{(1)}]$.   Any upper bound $X \in [B^{(0)},B^{(1)}]$ for $T'$ must have $X(m+1,1)=n+1$, so we should consider how we can start with $T'$ and increase the entry in cell $(m+1,1)$ from $b'$ to $n+1$.
\begin{itemize}
\item If $a'< b'$, then we can increase the entry in cell $(m+1,1)$ to $n+1$ without affecting any other entries of $T'$ to produce $T''=F_n \cdots F_{b'+1}F_{b'}(T')$---and any sequence of $F_i$ operators leading to $X$ could clearly be reordered to be prefaced by this initial sequence, so $T''$ is the least upper bound of $B^{(0)}$ and $T'$.  The join of $T'$ and $T$ can now be computed in $[B^{(0)},B^{(1)}]$ as $T''\vee T$.

\item Otherwise, $a'\geq b'$.  For $i=b',b'+1,\ldots,n$, in any path from $T'$ to $X$, we must at some point apply $F_i$ to increase the entry in cell $(m+1,1)$ from $i$ to $i+1$.  The only reason this might not work is if the entry in cell $(1,2)$ is also $i$, at which point the entry in cell $(1,2)$ would be increased.  So any upper bound $X \in [B^{(0)},B^{(1)}]$ with $X(m+1,1)=n+1$ must also have $X(1,2)=n+1$.  Define $T'' \in  [B^{(0)},B^{(1)}]$ by
\[T'' = F_n \cdots F_{b'+1} F_{b'} F_{n}\cdots F_{a'+1}F_{a'}(T')\]
so that \begin{align*}T''(i,j)&= \begin{cases}n+1 & \text{if } (i,j)=(m+1,1) \text{ or } (1,2) \\ T'(i,j) & \text{otherwise.} \end{cases}\end{align*} Then $T''$ is clearly minimal among all upper bounds $X$ for $T'$ with $X(m+1,1)=n+1$ and $X(1,2)=n+1$ (such upper bounds being ordered by componentwise comparison).
\end{itemize}
\end{itemize}

\subsection{\texorpdfstring{$\lambda=(k,1)$ or $(k^{n-1},k-1)$}{Wide hooks}}
By~\Cref{prop:dual}, the two crystals $\mathcal B_{(k,1)}^n$ and $\mathcal B_{(k^{n-1},k-1)}^n$ are dual, so we consider only $\lambda=(k,1)$.  The argument is almost identical to the one in the previous subsection.  Let $T,T'$ be given tableaux in $\mathcal B_\lambda^n$.  As in the previous subsection, it suffices to prove that $T$ and $T'$ have a join. If $n+1$ does not appear in $T$ or $T'$, then the interval from \[\raisebox{0.5\height}{\begin{ytableau}
1 & 1 & 1 & \cdots & 1 & 1 \\
2
\end{ytableau}}  \hspace{2em} \text{ to }  \hspace{2em} \raisebox{0.5\height}{\begin{ytableau}
\scalebox{.8}{{\it n}{-}1} & n & n & \cdots & n & n \\
n
\end{ytableau}}\]
in $\mathcal B_{(k,1)}^n$ matches the same interval in $\mathcal B_{(k,1)}^{n-1}$, and $T$ and $T'$ both appear within this interval.  By induction on $n$ (the base case $n=2$ is handled by~\Cref{sec:a1_and_a2} because $\mathcal B_{(k,1)}^2$ is a type $A_2$ crystal), this interval is a lattice, and we can find the join of $T$ and $T'$.  

Now assume that $n+1$ appears in $T$ or $T'$. Let $a=T(2,1)$, $a'=T'(2,1)$, $b=T(1,k)$, and $b'=T'(1,k)$.  At least one of $a,a',b,b'$ must be equal to $n+1$. By symmetry, we may assume that either $a$ or $b$ is $n+1$.

\begin{itemize}
    \item Suppose $a=n+1$ and $b,b'\neq n+1$. Consider the interval between \[A^{(0)}=\raisebox{0.5\height}{\begin{ytableau}
1 & 1 & 1 & \cdots & 1 & 1 \\
\scalebox{.8}{{\it n}{+}1}
\end{ytableau}} \hspace{2em}\text{ and }\hspace{2em} A^{(1)}=\raisebox{0.5\height}{\begin{ytableau}
n & n & n & \cdots & n & n \\
\scalebox{.8}{{\it n}{+}1}
\end{ytableau}}.\]
This interval $[A^{(0)},A^{(1)}]$ is isomorphic to $\mathcal B_{(k)}^{n-1}$ (and hence tableaux are ordered by component-wise comparison), so it is a lattice. Note that $T\in[A^{(0)},A^{(1)}]$.
\begin{itemize}
\item If $a'>b'$, then the tableau $T''=F_n \cdots F_{a'+1}F_{a'} (T')$ is the unique minimal element of the interval $[A^{(0)},A^{(1)}]$ greater than $T'$, so the least upper bound of $T$ and $T'$ is $T \vee T''$, the join of $T$ and $T''$ in $[A^{(0)},A^{(1)}]$.  By~\Cref{sec:mult_first_fund_weight}, this is also the tableau given by component-wise maximum of the first rows of $T$ and $T'$ that has entry $n+1$ in the cell $(2,1)$.

\item If $a'\leq b'$, then any upper bound $X$ for $A^{(0)}$ and $T'$ must have $X(2,1)=n+1$.  Let us consider how to start with $T'$ and increase the entry in cell $(2,1)$ from $a'$ to $n+1$.  For each $i=a',a'+1,\ldots,n$, in any path from $T'$ to $X$, we must at some point apply $F_i$ to increase the entry in cell $(2,1)$ from $i$ to $i+1$.  For this $F_i$ to have any effect on the entry in cell $(2,1)$, all entries equal to $i$ in the first row of the tableau will need to have already been increased to $i+1$ (no parenthesis pairing is possible with entries $i+1$, which all appear further to the right in the first row).  Thus, at some point, the entry in cell $(1,k)$ must be increased to $n+1$.  In fact, since $a'\leq b'$, we can increase the entry in cell $(1,k)$ to $n+1$ without affecting any other entries of $T'$ to produce $T''=F_n \cdots F_{b'+1}F_{b'}(T')$---and any sequence of $F_i$ operators leading to $X$ could be prefaced by this initial sequence.  We can therefore reduce to the case when $b'=n+1$. By symmetry, we can handle the case when $b=n+1$ instead.
\end{itemize}

\item If $b=n+1$, then we consider the interval between
\[B^{(0)}=\raisebox{0.5\height}{\begin{ytableau}
1 & 1 & 1 & \cdots & 1 & \scalebox{.8}{{\it n}{+}1} \\
2
\end{ytableau}} \hspace{2em}\text{ and }\hspace{2em} B^{(1)}=\raisebox{0.5\height}{\begin{ytableau}
n & \scalebox{.8}{{\it n}{+}1} & \scalebox{.8}{{\it n}{+}1} & \cdots & \scalebox{.8}{{\it n}{+}1} & \scalebox{.8}{{\it n}{+}1} \\
\scalebox{.8}{{\it n}{+}1}
\end{ytableau}}.\]
This interval $[B^{(0)},B^{(1)}]$ is isomorphic to $\mathcal B_{(k-1,1)}^n$, so it is a lattice by induction on $k$ (the case $k=1$ being taken care of by~\Cref{sec:minuscule}). Note that $T\in[B^{(0)},B^{(1)}]$.   Any upper bound $X$ for $T'$ in the interval must have $X(1,k)=n+1$, so we should consider how we can start at $T'$ and increase the entry in cell $(1,k)$ from $b'$ to $n+1$.
\begin{itemize}
\item If $a'\leq b'$, then we can increase the entry in cell $(1,k)$ to $n+1$ without affecting any other entries of $T'$ to produce $T''=F_n \cdots F_{b'+1}F_{b'}(T')$---and any sequence of $F_i$ operators leading to $X$ could be prefaced by this initial sequence, so $T''$ is the least upper bound of $B_0$ and $T'$.  The join of $T'$ and $T$ can now be computed in $[B^{(0)},B^{(1)}]$ as $T''\vee T$.

\item Otherwise, $a'>b'$.  Let $T' = T'_0 \xrightarrow{F_{t_1}} T'_1 \xrightarrow{F_{t_2}} \cdots \xrightarrow{F_{t_\ell}} T'_\ell = X$ be any path from $T'$ to the upper bound $X \in [B^{(0)},B^{(1)}]$.  Since $X(1,k)=n+1$, for each $i=b',b'+1,\ldots,n$, there is a step in this path $T'_{u_i} \xrightarrow{F_i} T'_{u_i+1}$ with $T'_{u_i}(1,k)=i$ and $T'_{u_i+1}(1,k)=i+1$.  For each $i$, in order to increase $T'_{u_i}(1,k)$ from $i$ to $i+1$, it must be the case that either $T'_{u_i}(2,1)>i+1$ or that $T'_{u_i}(2,1)=i+1$ and $T'_{u_i}(1,j_i)=i$ for some $j_i< k$.  
Note that for some $i$, we must have $T'_{u_i}(1,j_i)=T'_{u_i}(2,1)-1$ since the entry in cell $(2,1)$ cannot be increased past $n+1$.  Let $j$ be the smallest positive integer with $X(1,j) \geq a'-1$, noting that $j\leq k-1$ by the previous observation. Then $T'_r(1,j)<X(2,1)$ for all $r=0,1,\ldots,\ell$ since $T'_r(1,j)$ must pair with $T'_r(2,1)\geq a'$ before it can equal or exceed it.  Thus, $X(1,j)<X(2,1)$. 

Define $T'' \in [B^{(0)},B^{(1)}]$ by \[ \hspace{.25cm}T'' = \left(F_n \cdots F_{a'} F_{a'-1}\right) \left(F_{a'-2} \cdots F_{T'(1,k-1)+1} F_{T'(1,k-1)}\right) \left(F_{a'-2} \cdots F_{b'+1} F_{b'}\right)(T') \] so that  \[T''(i,j)= \begin{cases}a'-1 & \text{if } (i,j)=(1,k-1) \\ n+1 & \text{if } (i,j)=(1,k) \\ T'(i,j) & \text{otherwise.} \end{cases}\] 

We claim that the minimal upper bound for $T'$ in the interval $[B^{(0)},B^{(1)}]$ is the tableau $T''$.  For let $j$ be the smallest positive integer with $X(1,j) \geq a'-1$, so that $X(1,j) < X(2,1)$ as above.  From $T''$, we can construct a path to $X$ by first increasing the entry in every cell $(1,j),(1,j+1),\ldots,(1,k-2)$ in the first row of $T''$ to $a'-1$, then increasing the entries in cells $(1,j+1),\ldots,(1,k-2),(1,k-1)$ to their values in $X$ (as the entry in cell $(1,j)$ is now paired with that in cell $(2,1)$ when applying $F_{a'}$), then increasing the entries in cells $(1,j-1),\ldots,(1,2),(1,1)$ to their values in $X$, then increasing the entry in cell $(2,1)$ to its value in $X$, and finally increasing the entry in cell $(1,j)$ to its value in $X$.
\end{itemize}
\end{itemize}

\subsection{\texorpdfstring{$A_3$ and $\lambda=(3,2,1)$}{A3 and (3,2,1)}}
The exception of $\lambda=(3,2,1)$ for $A_3$ was checked by computer~\cite{Sage-Combinat}.

\ytableausetup{smalltableaux}
\section{Type A Crystals That Are Not Lattices}\label{sec:notlattices}
Let us fix an integer $n\geq 3$ and consider the type~$A_n$ crystal $\mathcal B_\lambda^n$ associated to a partition $\lambda=(\lambda_1,\ldots,\lambda_n,0)$. Throughout this section, we always let $\ell$ denote the number of positive parts of $\lambda$. We will often omit the trailing $0$'s and write $\lambda=(\lambda_1,\ldots,\lambda_\ell)$. The purpose of this section is to prove a sequence of lemmas that, when put together, demonstrate that $\mathcal B_\lambda^n$ is not a lattice if $n$ and $\lambda$ satisfy none of the conditions listed in \Cref{thm:main}. 

As a first step, we establish three lemmas that will allow us to simplify our notation later. If $(\lambda_1,\ldots,\lambda_\ell)$ is a partition and $1\leq a\leq b\leq \ell$, then we call the partition $(\lambda_a,\ldots,\lambda_b)$ a \defn{consecutive subpartition} of $\lambda$. 

\begin{lemma}\label{lem:removerows}
Let $\lambda=(\lambda_1,\ldots,\lambda_\ell)$ be a partition, and let $\mu$ be a consecutive subpartition of $\lambda$ that has $m$ parts. Suppose $n\geq \ell$. Then $\mathcal B_\mu^{n-\ell+m}$ is isomorphic to an interval in $\mathcal B_\lambda^n$. In particular, if $\mathcal B_\mu^{n-\ell+m}$ is not a lattice, then $\mathcal B_\lambda^n$ is not a lattice. 
\end{lemma}

\begin{proof}
It suffices to prove the lemma in the case when $\ell-m=1$; indeed, the general case will then follow by induction on $\ell-m$. Thus, we may assume $\mu$ is $(\lambda_2,\ldots,\ldots\lambda_\ell)$ or $(\lambda_1,\ldots,\lambda_{\ell-1})$. We will assume $\mu=(\lambda_2,\ldots,\lambda_\ell)$; the proof in the other case is very similar. 

Given a tableau $T\in\mathcal B_\mu^{n-1}$, let $\iota(T)$ be the tableau in $\mathcal B_\lambda^{n}$ satisfying $\iota(T)(1,j)=1$ for all $j\in[\lambda_1]$ and $\iota(T)(i,j)=T(i-1,j)+1$ for all $i\in\{2,\ldots,\ell\}$ and $j\in[\lambda_i]$. This defines a natural injection $\iota\colon\mathcal B_\mu^{n-1}\to\mathcal B_\lambda^n$. Note that $\iota(\mathcal B_\mu^{n-1})$ is equal to the set of tableaux in $\mathcal B_\lambda^n$ whose first rows only have $1$'s. It follows from the combinatorial description of the crystal partial order that $\iota(\mathcal B_{\mu}^{n-1})$ is an order ideal of $\mathcal B_\lambda^n$. To prove the desired result, we must show that for $T,T'\in\mathcal B_{\mu}^{n-1}$, we have $T\leq T'$ if and only if $\iota(T)\leq\iota(T')$.

Note that $\iota\circ F_i=F_{i+1}\circ\iota$. This observation implies that $\iota$ is order-preserving, meaning that if $T,T'\in\mathcal B_{\mu}^{n-1}$ are such that $T\leq T'$, then $\iota(T)\leq\iota(T')$. 

To prove the converse, assume $T,T'\in\mathcal B_\mu^{n-1}$ are such that $\iota(T)\leq\iota(T')$. Since $\iota(\mathcal B_{\mu}^{n-1})$ is an order ideal of $\mathcal B_\lambda^n$, there exist $T_1,\ldots,T_r$ such that $T=T_1$, $T'=T_r$, and $\iota(T_k)\lessdot\iota(T_{k+1})$ for all $k\in[r-1]$. Fix $k\in[r-1]$. Let $i\in[n]$ be such that $\iota(T_{k+1})=F_i(\iota(T_k))$. Because $\iota(T_k)$ and $\iota(T_{k+1})$ both have only $1$'s in their first rows, we know that $i\geq 2$. We know that $F_i\circ\iota=\iota\circ F_{i-1}$, so $\iota(T_{k+1})=\iota(F_{i-1}(T_k))$. Hence, $T_{k+1}=F_{i-1}(T_k)$. This shows that $T_k\lessdot T_{k+1}$. As $k$ was arbitrary, we must have $T_1\lessdot\cdots\lessdot T_r$. Therefore, $T\leq T'$. 
\end{proof}

\begin{lemma}\label{lem:removecolumns}
Let $\lambda=(\lambda_1,\ldots,\lambda_\ell)$ be a partition, and consider a nonnegative integer $t< \lambda_\ell$. Let $\mu=(\lambda_1-t,\ldots,\lambda_\ell-t)$ be the partition whose Young diagram is obtained by deleting the first $t$ columns from that of $\lambda$. Suppose $n\geq \ell$. Then $\mathcal B_\mu^{n}$ is isomorphic to a principal order ideal in $\mathcal B_\lambda^n$. In particular, if $\mathcal B_\mu^{n}$ is not a lattice, then $\mathcal B_\lambda^n$ is not a lattice. 
\end{lemma}

\begin{proof}
Given a tableau $T\in\mathcal B_\mu^{n}$, let $\eta(T)$ be the tableau in $\mathcal B_\lambda^{n}$ satisfying $\eta(T)(i,j)=i$ for all $i\in[\ell]$ and $j\in[t]$ and satisfying $\iota(T)(i,j)=T(i,j-t)$ for all $i\in[\ell]$ and $j\in\{t+1,\ldots,\lambda_i\}$. This defines a natural injection $\eta\colon\mathcal B_\mu^{n}\to\mathcal B_\lambda^n$. Note that a tableau in $\mathcal B_\lambda^n$ is in the image of $\eta$ if and only if it contains the entry $i$ in cell $(i,j)$ for every $i\in[\ell]$ and every $j\in[t]$. It follows from the combinatorial description of the crystal partial order that $\eta(\mathcal B_{\mu}^{n})$ is a principal order ideal of $\mathcal B_\lambda^n$. To prove the desired result, we must show that for $T,T'\in\mathcal B_{\mu}^{n}$, we have $T\leq T'$ if and only if $\eta(T)\leq\eta(T')$.

Observing that $\eta\circ F_i=F_{i}\circ\eta$ for all $i$, we see that $\eta$ is order-preserving. Now assume $T,T'\in\mathcal B_\mu^{n}$ are such that $\eta(T)\leq\eta(T')$. Since $\eta(\mathcal B_{\mu}^{n})$ is an order ideal of $\mathcal B_\lambda^n$, there exist $T_1,\ldots,T_r$ such that $T=T_1$, $T'=T_r$, and $\eta(T_k)\lessdot\eta(T_{k+1})$ for all $k\in[r-1]$. Fix $k\in[r-1]$. There exists $i\in[n]$ such that $\eta(T_{k+1})=F_i(\eta(T_k))=\eta(F_i(T_k))$. Then $T_{k+1}=F_{i}(T_k)$, so $T_k\lessdot T_{k+1}$. As $k$ was arbitrary, we must have $T_1\lessdot\cdots\lessdot T_r$. Therefore, $T\leq T'$. 
\end{proof}

The next lemma follows from the fact, which is straightforward to verify, that $\mathcal B_\lambda^n$ is principal order ideal of $\mathcal B_\lambda^{n+1}$. 

\begin{lemma}\label{lem:largern}
Let $\lambda$ be a partition with at most $n$ parts. Suppose $n'\geq n$. If $\mathcal B_\lambda^n$ is not a lattice, then $\mathcal B_\lambda^{n'}$ is not a lattice. 
\end{lemma}

The main idea used in each of the following lemmas is to prove that a crystal $\mathcal B_\lambda^n$ is not a lattice by finding a small ``bowtie'' inside of it.\footnote{We thank Jon McCammond for introducing us to this coutural terminology.} More precisely, we will exhibit distinct tableaux $T_1,T_2,U_1,U_2\in\mathcal B_\lambda^n$ satisfying the following conditions: 
\begin{equation}\label{eq:bowtie}
\begin{aligned}
&T_1\text{ and }T_2\text{ are incomparable and }U_1\text{ and }U_2\text{ are incomparable;} \\ &T_1\lessdot U_1,\quad T_1\leq U_2,\quad T_2\leq U_1,\quad\text{and}\quad T_2\leq U_2.
\end{aligned}
\end{equation}
These conditions immediately imply that $T_1$ and $T_2$ do not have a join in $\mathcal B_\lambda^n$. (It is important that $U_1$ covers $T_1$.)

In what follows, recall that $\lambda=(\lambda_1,\ldots,\lambda_\ell)$ is a partition with $\ell$ parts, where $\ell\leq n$.

\begin{lemma}\label{lem:A}
If there exists $p\in[\ell-2]$ such that $\lambda_p>\lambda_{p+1}\geq\lambda_{p+2}\geq 2$, then $\mathcal B_\lambda^n$ is not a lattice. 
\end{lemma}

\begin{proof}
By \Cref{lem:removerows}, it suffices to prove the result when $\ell=3$. Assuming $\ell=3$, we can then use \Cref{lem:removecolumns} to see that we may further assume $\lambda_3=2$. Thus, assume $\lambda=(\lambda_1,\lambda_2,2)$, where $\lambda_1>\lambda_2$. Consider the tableaux $T_1,T_2,U_1,U_2$ of shape $\lambda$ defined as follows:
\begin{itemize}
    \item $T_1$ has rows $(1^{\lambda_2-1},2,2,4^{\lambda_1-\lambda_2-1})$, $(2^{\lambda_2-1},3)$, $(3,4)$;
    \item $T_2$ has rows $(1^{\lambda_2},3,4^{\lambda_1-\lambda_2-1})$, $(2^{\lambda_2})$, $(3,4)$;
    \item $U_1$ has rows $(1^{\lambda_2-1},2,3,4^{\lambda_1-\lambda_2-1})$, $(2^{\lambda_2-1},3)$, $(3,4)$; 
    \item $U_2$ has rows $(1^{\lambda_2-1},2,3,4^{\lambda_1-\lambda_2-1})$, $(2^{\lambda_2-1},3)$, $(4,4)$. 
\end{itemize} These tableaux satisfy the conditions in \eqref{eq:bowtie}. Indeed, it is straightforward to check that $T_1$ and $T_2$ are incomparable and that $U_1$ and $U_2$ are incomparable. We also have $U_1=F_2(T_1)$, $U_1=F_1F_2(T_2)$, $U_2=F_2F_3(T_1)$, and $U_2=F_1F_2F_3(T_2)$. For example, \Cref{fig:A} illustrates the relevant subposet of the Hasse diagram of $\mathcal B_\lambda^n$ when $\lambda=(5,2,2)$.
\end{proof}

\begin{figure}[ht]
\begin{tikzpicture}[scale=3]
\node[label=left:{$U_1=$}] (123442334) at (0,1.6) {\begin{ytableau}
1 & 2 & 3 & 4 & 4 \\
2 & 3 \\
3 & 4
\end{ytableau}};
\node[label=left:{$T_1=$}] (122442334) at (0,0.8) {\begin{ytableau}
1 & 2 & 2 & 4 & 4 \\
2 & 3 \\
3 & 4
\end{ytableau}};
\node[label=left:{$U_2=$}] (123442344) at (2.5,2.4) {\begin{ytableau}
1 & 2 & 3 & 4 & 4 \\
2 & 3 \\
4 & 4
\end{ytableau}};
\node (113442344) at (2.5,1.6) {\begin{ytableau}
1 & 1 & 3 & 4 & 4 \\
2 & 3 \\
4 & 4
\end{ytableau}};
\node (113442244) at (2.5,0.8) {\begin{ytableau}
1 & 1 & 3 & 4 & 4 \\
2 & 2 \\
4 & 4
\end{ytableau}};
\node[label=left:{$T_2=$}] (113442234) at (2.5,0) {\begin{ytableau}
1 & 1 & 3 & 4 & 4 \\
2 & 2 \\
3 & 4
\end{ytableau}};
\node (113442334) at (1.25,0.8) {\begin{ytableau}
1 & 1 & 3 & 4 & 4 \\
2 & 3 \\
3 & 4
\end{ytableau}};
\node (122442344) at (1.25,1.6) {\begin{ytableau}
1 & 2 & 2 & 4 & 4 \\
2 & 3 \\
4 & 4
\end{ytableau}};
\draw [->,blue] (122442334) -- (123442334)  node [midway, left] {$F_2$};
\draw [->,red] (113442344) -- (123442344)  node [midway, right] {$F_1$};
\draw [->,blue] (113442244) -- (113442344)  node [midway, right] {$F_2$};
\draw [->,green] (113442234) -- (113442244)  node [midway, right] {$F_3$};
\draw [->,blue] (113442234) -- (113442334)  node [midway, right,xshift=1em] {$F_2$};
\draw [->,red] (113442334) -- (123442334)  node [pos=0.3, right,xshift=0.5em] {$F_1$};
\draw [->,green] (122442334) -- (122442344)  node [pos=0.3, left,xshift=-1em] {$F_3$};
\draw [->,blue] (122442344) -- (123442344)  node [midway, left,xshift=-1em] {$F_2$};
\end{tikzpicture}
\caption{Illustration of the proof of \Cref{lem:A}.}
\label{fig:A}
\end{figure}
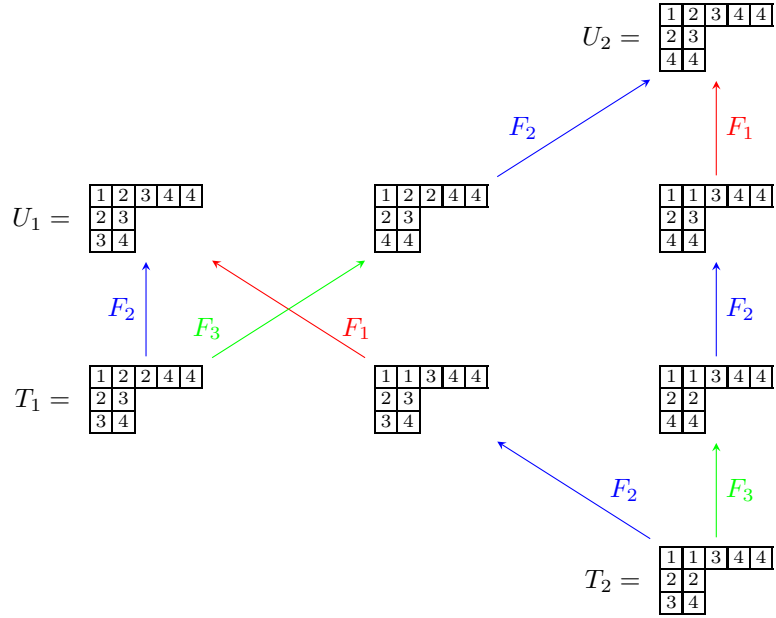

\begin{lemma}\label{lem:B}
If there exists $p\in[\ell-2]$ such that $\lambda_p-2\geq\lambda_{p+1}\geq \lambda_{p+2}$, then $\mathcal B_\lambda^n$ is not a lattice. 
\end{lemma}

\begin{proof}
By \Cref{lem:removerows}, it suffices to prove the result when $\ell=3$. Assuming $\ell=3$, we can then use \Cref{lem:removecolumns} to see that we may further assume $\lambda_3=1$. Thus, assume $\lambda=(\lambda_1,\lambda_2,1)$, where $\lambda_1-2\geq\lambda_2$. Consider the tableaux $T_1,T_2,U_1,U_2$ of shape $\lambda$ defined as follows:
\begin{itemize}
    \item $T_1$ has rows $(1,1,3^{\lambda_2-1},4^{\lambda_1-\lambda_2-1})$, $(2,4^{\lambda_2-1})$, $(4)$;
    \item $T_2$ has rows $(1,1,3^{\lambda_2},4^{\lambda_1-\lambda_2-2})$, $(2,4^{\lambda_2-1})$, $(4)$;
    \item $U_1$ has rows $(1,1,3^{\lambda_2-1},4^{\lambda_1-\lambda_2-1})$, $(3,4^{\lambda_2-1})$, $(4)$; 
    \item $U_2$ has rows $(1,3^{\lambda_2},4^{\lambda_1-\lambda_2-1})$, $(2,4^{\lambda_2-1})$, $(4)$. 
\end{itemize} These tableaux satisfy the conditions in \eqref{eq:bowtie}. Indeed, it is straightforward to check that $T_1$ and $T_2$ are incomparable and that $U_1$ and $U_2$ are incomparable. We also have $U_1=F_2(T_1)$, $U_1=F_3F_2(T_2)$, $U_2=F_2F_1(T_1)$, and $U_2=F_3F_2F_1(T_2)$. For example, \Cref{fig:B} illustrates the relevant subposet of the Hasse diagram of $\mathcal B_\lambda^n$ when $\lambda=(6,3,1)$.
\end{proof}

\begin{figure}[ht]
\begin{tikzpicture}[scale=3]
\node[label=left:{$U_1=$}] (123442334) at (0,1.6) {\begin{ytableau}
1 & 1 & 3 & 3 & 4 & 4 \\
3 & 4 & 4 \\
4
\end{ytableau}};
\node[label=left:{$T_1=$}] (122442334) at (0,0.8) {\begin{ytableau}
1 & 1 & 3 & 3 & 4 & 4 \\
2 & 4 & 4 \\
4
\end{ytableau}};
\node[label=left:{$U_2=$}] (123442344) at (2.5,2.4) {\begin{ytableau}
1 & 3 & 3 & 3 & 4 & 4 \\
2 & 4 & 4 \\
4
\end{ytableau}};
\node (113442344) at (2.5,1.6) {\begin{ytableau}
1 & 3 & 3 & 3 & 3 & 4 \\
2 & 4 & 4 \\
4
\end{ytableau}};
\node (113442244) at (2.5,0.8) {\begin{ytableau}
1 & 2 & 3 & 3 & 3 & 4 \\
2 & 4 & 4 \\
4
\end{ytableau}};
\node[label=left:{$T_2=$}] (113442234) at (2.5,0) {\begin{ytableau}
1 & 1 & 3 & 3 & 3 & 4 \\
2 & 4 & 4 \\
4
\end{ytableau}};
\node (113442334) at (1.25,0.8) {\begin{ytableau}
1 & 1 & 3 & 3 & 3 & 4 \\
3 & 4 & 4 \\
4
\end{ytableau}};
\node (122442344) at (1.25,1.6) {\begin{ytableau}
1 & 2 & 3 & 3 & 4 & 4 \\
2 & 4 & 4 \\
4
\end{ytableau}};
\draw [->,blue] (122442334) -- (123442334)  node [midway, left] {$F_2$};
\draw [->,green] (113442344) -- (123442344)  node [midway, right] {$F_3$};
\draw [->,blue] (113442244) -- (113442344)  node [midway, right] {$F_2$};
\draw [->,red] (113442234) -- (113442244)  node [midway, right] {$F_1$};
\draw [->,blue] (113442234) -- (113442334)  node [midway, right,xshift=1em] {$F_2$};
\draw [->,green] (113442334) -- (123442334)  node [pos=0.3, right,xshift=0.5em] {$F_3$};
\draw [->,red] (122442334) -- (122442344)  node [pos=0.3, left,xshift=-1em] {$F_1$};
\draw [->,blue] (122442344) -- (123442344)  node [midway, left,xshift=-1em] {$F_2$};
\end{tikzpicture}
\caption{Illustration of the proof of \Cref{lem:B}.}
\label{fig:B}
\end{figure}
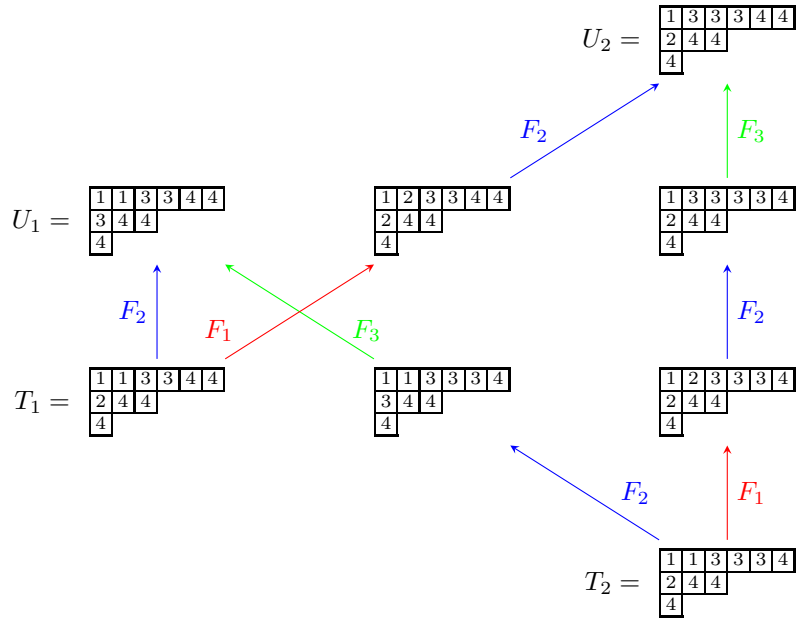

\begin{lemma}\label{lem:C}
If there exists $p\in[\ell-2]$ such that $\lambda_p\geq\lambda_{p+1}\geq\lambda_{p+2}+2$, then $\mathcal B_\lambda^n$ is not a lattice. 
\end{lemma}

\begin{proof}
For this proof, we write $\lambda=(\lambda_1,\ldots,\lambda_{n+1})$ with $\lambda_{i}=0$ for all $i\in\{\ell+1,\ldots,n+1\}$ (in particular, $\lambda_{n+1}=0$). Let $\lambda^* = (\lambda_1^*,\ldots,\lambda_{n+1}^*)$, where $\lambda_i^*=\lambda_1-\lambda_{n+2-i}$. \Cref{prop:dual} tells us that $\mathcal B_{\lambda^*}^n$ is isomorphic to the dual crystal of $\mathcal B_\lambda^n$. The underlying poset structures of $\mathcal B_\lambda^n$ and its dual are isomorphic, so $\mathcal B_\lambda^n$ and $\mathcal B_{\lambda^*}^n$ are isomorphic as posets. The hypothesis of the lemma translates into the inequalities $\lambda_{n-p}^*-2\geq\lambda_{n-p+1}^*\geq\lambda_{n-p+2}^*$, so it follows from \Cref{lem:B} that $\mathcal B_{\lambda^*}^n$ is not a lattice.
\end{proof}

\begin{lemma}\label{lem:D}
If $\ell\geq 4$ and $\lambda$ is of the form $(3^m,2,1^{\ell-m-1})$ for some $m\in[\ell-2]$, then $\mathcal B_\lambda^n$ is not a lattice.  
\end{lemma} 

\begin{proof}
By \Cref{lem:removerows}, it suffices to prove the lemma when $\lambda$ is $(3,3,2,1)$ or $(3,2,1,1)$. By \Cref{lem:largern}, it suffices to prove that $\mathcal B_{(3,3,2,1)}^4$ and $\mathcal B_{(3,2,1,1)}^4$ are not lattices. This is straightforward to do by computer or by hand. Indeed, the tableaux \[\raisebox{\height}{\begin{ytableau}
1 & 2 & 2 \\
3 & 3 & 4 \\
4 & 5 \\
5
\end{ytableau}}\quad\text{and} \quad\raisebox{\height}{\begin{ytableau}
1 & 2 & 3 \\
3 & 3 & 4 \\
4 & 5 \\
5
\end{ytableau}}\] have no join in $\mathcal B_{(3,3,2,1)}^4$, while the tableaux \[\raisebox{\height}{\begin{ytableau}
1 & 1 & 3 \\
2 & 5 \\
4 \\
5
\end{ytableau}}\quad\text{and} \quad\raisebox{\height}{\begin{ytableau}
1 & 1 & 4 \\
2 & 5 \\
4 \\
5
\end{ytableau}}\] have no join in $\mathcal B_{(3,2,1,1)}^4$. 
\end{proof}

\begin{lemma}\label{lem:E}
If $2\leq \ell<n$ and $\lambda_2\geq 2$, then $\mathcal B_\lambda^n$ is not a lattice. 
\end{lemma}

\begin{proof}
By \Cref{lem:removerows}, it suffices to prove the statement when $\ell=2$. Assuming $\ell=2$, we can then use \Cref{lem:removecolumns} to see that we may further assume $\lambda_2=2$. Thus, assume $\lambda=(\lambda_1,2)$. Consider the tableaux $T_1,T_2,U_1,U_2$ of shape $\lambda$ defined as follows:
\begin{itemize}
    \item $T_1$ has rows $(1^{\lambda_1-1},3)$ and $(3,4)$;
    \item $T_2$ has rows $(1^{\lambda_1-1},2)$ and $(3,4)$;
    \item $U_1$ has rows $(1^{\lambda_1-1},3)$ and $(4,4)$; 
    \item $U_2$ has rows $(1^{\lambda_1-2},2,3)$ and $(3,4)$. 
\end{itemize} These tableaux satisfy the conditions in \eqref{eq:bowtie}. Indeed, it is straightforward to check that $T_1$ and $T_2$ are incomparable and that $U_1$ and $U_2$ are incomparable. We also have $U_1=F_3(T_1)$, $U_1=F_2F_3(T_2)$, $U_2=F_1(T_1)$, and $U_2=F_2F_1(T_2)$. For example, \Cref{fig:E} illustrates the relevant subposet of the Hasse diagram of $\mathcal B_\lambda^n$ when $\lambda=(5,2)$.
\end{proof}

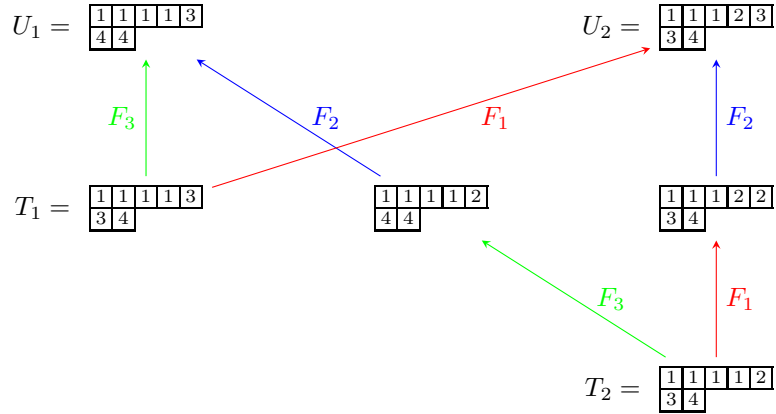
\begin{figure}[ht]
\begin{tikzpicture}[scale=3]
\node[label=left:{$U_1=$}] (123442334) at (0,1.6) {\begin{ytableau}
1 & 1 & 1 & 1 & 3 \\
4 & 4
\end{ytableau}};
\node[label=left:{$T_1=$}] (122442334) at (0,0.8) {\begin{ytableau}
1 & 1 & 1 & 1 & 3 \\
3 & 4
\end{ytableau}};
\node[label=left:{$U_2=$}] (123442344) at (2.5,1.6) {\begin{ytableau}
1 & 1 & 1 & 2 & 3 \\
3 & 4
\end{ytableau}};
\node (113442244) at (2.5,0.8) {\begin{ytableau}
1 & 1 & 1 & 2 & 2 \\
3 & 4
\end{ytableau}};
\node[label=left:{$T_2=$}] (113442234) at (2.5,0) {\begin{ytableau}
1 & 1 & 1 & 1 & 2 \\
3 & 4
\end{ytableau}};
\node (113442334) at (1.25,0.8) {\begin{ytableau}
1 & 1 & 1 & 1 & 2 \\
4 & 4
\end{ytableau}};
\draw [->,blue] (113442244) -- (123442344)  node [midway, right] {$F_2$};
\draw [->,red] (113442234) -- (113442244)  node [midway, right] {$F_1$};
\draw [->,blue] (113442334) -- (123442334)  node [midway, right,xshift=.5em] {$F_2$};
\draw [->,green] (113442234) -- (113442334)  node [midway, right,xshift=.5em] {$F_3$};
\draw [->,red] (122442334) -- (123442344)  node [midway, right,xshift=1.5em] {$F_1$};
\draw [->,green] (122442334) -- (123442334)  node [midway, left] {$F_3$};
\end{tikzpicture}
\caption{Illustration of the proof of \Cref{lem:E}.}
\label{fig:E}
\end{figure}

We can now combine the preceding lemmas to complete the proof of \Cref{thm:main}. 

\begin{proposition}
Let $n\geq 3$, and let $\lambda=(\lambda_1,\ldots,\lambda_\ell)$ be a partition with at most $n$ parts such that $n$ and $\lambda$ satisfy none of the conditions listed in \Cref{thm:main}. The crystal $\mathcal B_\lambda^n$ is not a lattice. 
\end{proposition}

\begin{proof}
Suppose by way of contradiction that $\mathcal B_\lambda^n$ is a lattice. The partitions of the form $(k)$ are listed in \Cref{thm:main}, so we must have $\ell\geq 2$. If $\lambda_2=1$, then we must have $\lambda_1\geq 3$ and $\ell\geq 3$ since the partitions of the form $(k,1)$ or $(2,1^m)$ are listed in \Cref{thm:main}. However, this contradicts \Cref{lem:B}. Consequently, we must have $\lambda_2\geq 2$. It now follows from \Cref{lem:E} that $\ell=n$. 

Let $m$ be the number of parts in $\lambda$ that are equal to $1$. Because the partition $(1^m)$ is listed in \Cref{thm:main}, we must have $m\leq n-1$. It follows from \Cref{lem:A} that $\lambda$ must be of the form $(k^{n-m},1^m)$ for some $k\geq 2$ or of the form $(r^{n-m-1},k,1^m)$ for some $r>k\geq 2$. Suppose for the moment that $\lambda=(k^{n-m},1^m)$ for some $k\geq 2$. Because the partition $(2^{n-m},1^m)$ is listed in \Cref{thm:main}, we know that $k\geq 3$. We must have $m=n-1$ since, otherwise, we could set $p=n-m-1$ in \Cref{lem:C} and reach a contradiction. Thus, $\lambda=(k,1^{n-1})$. However, this is a contradiction because we saw above that $\lambda_2\geq 1$. 

We deduce from the previous paragraph that $\lambda=(r^{n-m-1},k,1^m)$ for some $r>k\geq 2$. If we had $r-2\geq k$, then we could set $p=n-m-1$ in \Cref{lem:B} to reach a contradiction. This demonstrates that $r=k+1$. Since $((k+1)^{n-1},k)$ is one of the partitions listed in \Cref{thm:main}, $m$ cannot be $0$. If we had $k\geq 3$, then we could set $p=n-m-1$ in \Cref{lem:C} to obtain a contradiction, so we must have $k=2$ and $r=3$. Since $\ell=n$, \Cref{lem:D} now forces us to have $n\leq 3$. This implies that $\lambda=(3,2,1)$ and $n=\ell=3$. However, this is one of the cases listed in \Cref{thm:main}, so we have reached our final contradiction. 
\end{proof}

\section*{Acknowledgements}
N.W. was partially supported by a Simons Foundation Collaboration Grant. C.D. was supported by a Fannie and John Hertz Foundation Fellowship and an NSF Graduate Research Fellowship (grant number DGE-1656466).

\bibliographystyle{amsalpha}
\bibliography{crystal_lattices}

\end{document}